\documentclass[a4paper]{article}

\usepackage[utf8]{inputenc}
\usepackage[english]{babel}
\usepackage{amsmath, amsthm, amsfonts, amssymb, enumerate}
\usepackage{mathrsfs}
\usepackage[numbers]{natbib}
\usepackage{xcolor}
\usepackage{graphicx}
\usepackage{hyperref}
\usepackage[normalem]{ulem}
\usepackage{soul}
\usepackage{bm}
\usepackage{bbm} 
\usepackage{verbatim}   
\usepackage{ifthen}
\usepackage{hyperref}

\theoremstyle{plain}

\newtheorem{lemma}{Lemma}
\newtheorem{theorem}[lemma]{Theorem}
\newtheorem*{lemma*}{Lemma}
\newtheorem{corollary}[lemma]{Corollary}
\newtheorem*{theorem*}{Theorem}
\newtheorem*{corollary*}{Corollary}
\newtheorem{proposition}[lemma]{Proposition}
\newtheorem*{proposition*}{Proposition}

\theoremstyle{definition}
\newtheorem{definition}[lemma]{Definition}
\newtheorem{example}[lemma]{Example}
\newtheorem{remark}[lemma]{Remark}
\newtheorem*{remark*}{Remark}
\newtheorem*{notation*}{Notation}
\newtheorem*{convention*}{Convention}

\definecolor{aquam}{rgb}{0.5,1.0,1.0}
\definecolor{bbrown}{rgb}{0.75,0.38,0.15}
\definecolor{Cyan}{rgb}{0,0.6,0.6}
\definecolor{Darkblue}{rgb}{0,0,1}
\definecolor{Dodgerblue2}{rgb}{0,0.5,1}
\definecolor{Green}{rgb}{0,0.6,0.06}
\definecolor{Kahki}{rgb}{1,1,0.5}
\definecolor{Magenta}{rgb}{0.7,0,0.7}
\definecolor{bMagenta}{rgb}{1,.6,1}
\definecolor{Orange}{rgb}{0.8,0.3,0}
\definecolor{dOrchid}{rgb}{0.7,0.2,0.4}
\definecolor{Orchid}{rgb}{1,0.5,1}
\definecolor{Purple}{rgb}{0.65,0.07,0.85}
\definecolor{Royalblue}{rgb}{0.6,0.85,0.87}
\definecolor{Tan}{rgb}{0.54,0.42,0.23}
\definecolor{bTan}{rgb}{0.94,0.82,0.63}
\definecolor{zoltan}{rgb}{0,0.1,0.3}
\definecolor{Turquoise}{rgb}{0,0.85,0.87}
\definecolor{Yellow}{rgb}{1,1,0}
\definecolor{darkamber}{rgb}{0.4,0.19,0.28}
\definecolor{bYellow}{rgb}{1,1,0.6}
\definecolor{bRed}{rgb}{1,0.7,0.7}
\definecolor{boxcolb}{rgb}{0.87,0.77,0.75}
\definecolor{boxcol}{rgb}{0.6,0.85,0.87}
\definecolor{boxcolgreen}{rgb}{0.64,0.93,0.79}
\definecolor{boxcolaa}{rgb}{.75,.99,.70}
\definecolor{boxcolbb}{rgb}{0.39,0.50,0.56}
\definecolor{boxcolcc}{rgb}{1,0.81,0.65}
\definecolor{yy}{rgb}{0.43,0.21,.18}
\definecolor{gA}{gray}{0.5}
\definecolor{gB}{gray}{0.8}
\definecolor{gC}{gray}{0.9}

\usepackage{xcolor}

\hypersetup{
  colorlinks   = true, 
  urlcolor     = blue, 
  linkcolor    = blue, 
  citecolor   = red 
}

\newcommand{\conv}{\operatorname{conv}}

\newcommand{\N}{{\mathbb N}}

\newcommand{\R}{{\mathbb R}}
\newcommand{\Q}{{\mathbb Q}}

\newcommand{\msr}{Y}

\newcommand{\pr}{\mathrm{pr}}
\newcommand{\hdim}{\mathrm{dim}_H}
\newcommand{\cgraph}{\operatorname{cgraph}}
\newcommand{\graph}{\operatorname{graph}}

\definecolor{grey}{rgb}{0.7, 0.7, 0.7}
\definecolor{darkgreen}{rgb}{0., 0.6, 0.}

\newcommand{\cantorf}{\mathfrak{C}}

\newcommand{\exclude}[1]{}

\title{Continuous functions with impermeable graphs}

\author{Zoltán Buczolich\thanks{ The project leading to this application has received funding from the European Research Council (ERC) under the European Union’s Horizon 2020 research and innovation programme (grant agreement No. 741420).
This author was also supported by the Hungarian National Research, Development and Innovation Office--NKFIH, Grant 124749 and  at the time of completion of this paper was holding a visiting researcher position 
at the Rényi Institute.}, Gunther Leobacher\thanks{G.~Leobacher is supported by the Austrian
Science Fund (FWF): Project F5508-N26, which is part of the Special Research
Program `Quasi-Monte Carlo Methods: Theory and Applications'.}\; and Alexander Steinicke}

\begin{document}

\maketitle

\begin{abstract}
We construct a Hölder continuous function on the unit interval 
which coincides  in uncountably (in fact continuum)  many points with every function of total variation smaller than 1 passing  through the origin.

We say that a function with this property has impermeable graph, and we
present 
further examples of functions both with permeable and impermeable graphs.

The first example function is subsequently used to construct an example of a continuous function on the plane 
which is intrinsically Lipschitz continuous on the complement of the graph of a 
Hölder continuous function
with impermeable graph, but which is not Lipschitz continuous on the plane. 

As another main result we construct a continuous function on the unit interval 
which coincides  in a set of Hausdorff dimension 1 with every function of total variation smaller than 1 which passes through the origin.

\end{abstract}

\bigskip



\section{Introduction}

Although continuous functions on compact intervals are among  the most prominent mathematical objects studied
in undergraduate analysis courses, they continue to offer many interesting questions and results on an advanced level.
One message of historical importance is that general continuous functions differ significantly from differentiable ones, 
 in that they may be nowhere differentiable and  nowhere locally of bounded variation.

The differences become less expressed if one compares continuous functions with Lipschitz- or Hölder continuous ones. 
For example, it is well-known that with respect to the Wiener measure, almost every continuous function is Hölder continuous, but  is also almost nowhere differentiable and of infinite total variation, see \cite[Theorem 9.25, Theorem 9.18, Exercise 5.21]{karatzas-shreve} or 
\cite[Chapter I, Theorem 2.1, Exercise 2.9, Corollary 2.5]{revuz2004continuous}.  
In this paper, we introduce a new property of functions, namely, that of permeability
of its graph. This notion derives from a geometric notion of permeability of subsets of metric spaces, 
introduced in \cite{leoste21}, which we will recall in Section \ref{sec:cantor-madness}, and which is related to concepts of removability of subsets, cf.~\cite{rajala2019,younsi2015}.

Loosely speaking, a function has a permeable graph, if its graph can be avoided by the graph of some bounded variation
function to the degree that the graphs have at most countably  many intersection
points.

 For a function $f\colon [a,b]\to \R$ let
$V_a^b(f)$ denote its  (possibly  infinite) total variation over the interval $[a,b]$.

\begin{definition}\label{def:permeable-graph}
Let $g\colon [a,b]\to \R$ be a function. We say that $g$ has a {\em permeable graph}, if for all 
$y\in \R$ and all $\delta>0$ there exists a function $f\colon [a,b]\to \R$ with 
$f(a)=f(b)=y$,  $V_a^b(f)<\delta$
and such that the topological closure of the set
\[
 \{t\in [a,b]\colon f(t)=g(t)\}
\]
is at most countable. If a function does not have a permeable graph, we say that it has an {\em impermeable graph}.
\end{definition}

If $g$ has an impermeable graph then there exist $y\in \R$ and a $\delta>0$  such that for any $f\colon [a,b]\to \R$ with 
$f(a)=f(b)=y$, $V_a^b(f)<\delta$, 
the topological closure  
$\overline{\{t\in [a,b]\colon f(t)=g(t)\}}$  is uncountable. Then  this set    is an uncountable Borel set
and hence by the Perfect Set Theorem for Borel Sets \cite[13.6 Theorem]{kechris1995classical} it contains a Cantor set and therefore  is of cardinality continuum.

For example, it is immediate that polynomials have permeable graphs: Let $g$ be a polynomial of degree $n$,
and let $\delta>0$. If $n\ge 1$, then $g$ has at most $n$ intersections with any line, so given 
$y\in \R$, the graph of the function $f$ with  $f(t)=y$, $t\in [a,b]$, has at most $n$ intersections with that of $g$, 
and $V_a^b(f)=0<\delta$. Thus $g$ has a permeable graph.
If $n=0$ or $g\equiv 0$, we distinguish two cases. If $g(0)\ne y$ set $f(t)=y$, $t\in [a,b]$, so that $V_a^b(f)=0$ and the graphs of $f$ and $g$ have no intersection. Otherwise set
\[
f(t):=\begin{cases}
      y& \text{if }t\in \{a,b\}\\
      y+\frac{\delta}{3} & \text{if }t=1\,.
     \end{cases}
\]
Then 
the graphs of $f$ and $g$ have precisely two intersections and $V_a^b(f)< \delta$.

Therefore the graph of $g$ is permeable.\\

In the preceding example we were even able to find a function $f$ with only finitely many intersections with the graph of $g$. In that case we say that $g$ has a {\em finitely permeable graph}. There is a significant generalization of 
this  example:

\begin{theorem}\label{thm:bdd}
Let $g\colon [a,b]\to \R$ be of bounded variation. Then $g$ has a finitely permeable graph.
\end{theorem}

\begin{proof}
Let $y\in \R$ and $\delta>0$. 
By Banach's  indicatrix theorem \cite[Theorème 2 or Corollaire 1]{banach25}, $g^{-1}(\{z\})$ is finite for Lebesgue-a.e.~$z\in\R$.
Therefore we can find  $y_0\in (y-\tfrac{\delta}{2},y+\tfrac{\delta}{2})$ such that  $n:=\# g^{-1}(\{y_0\})$ is finite.

Define $f\colon [a,b]\to \R$ by 
\[
f(x):=\begin{cases}
y&\text{if } x\in \{a,b\}\,,\\
y_0&\text{else.}
\end{cases}
\]
Then $V_a^b(f)<\delta$ and, by construction,  $\{x\colon f(x)=g(x)\}\le n+2$, since it might be that 
$g(a)=y$ and/or $g(b)=y$.
\end{proof}

The question about permeability of graphs stands in a long tradition of studying the zero and level sets of 
continuous functions. For example, in \cite{benavides} it  was   shown that for the ``typical'' continuous function on 
$[0,1]$, which has at least one zero, the set of zeros is uncountable but has Lebesgue measure 0.
Banach \cite{banach25} studied the cardinality of the intersection of the graph of a bounded variation function 
with horizontal lines. This led to the famous Banach's indicatrix theorem. \v{C}ech \cite{Cech31} proved that 
a continuous function on a compact interval for which the intersection with every  horizontal line is finite, must be monotone on some
subinterval. 
In \cite{bruckner93} a number of extensions of \v{C}ech's theorem are shown.  In \cite{Buczolich88} it is shown 
that every continuous function is either convex or concave on some subinterval, or there are 2 nonempty perfect
sets such that the restriction of the function on the first set is strictly convex and that on the second set is strictly concave.
In \cite{minakshisundar} it is shown that almost every horizontal line that intersects the graph of a continuous nowhere differentiable function $g$  does so in an uncountable set. That statement obviously 
transfers to the intersection with (non-vertical) lines or, more generally, with the graph of a 
differentiable function. Another result in this direction is \cite[Theorem 12]{Brown1999}, which states that 
for a nowhere monotone continuous function $g$, the level sets $g^{-1}(\{y\})$ are uncountable if $y\in (\min(g),\max(g))$. Our result fits into this line of research in that we present a continuous function which has
the property that every continuous function of sufficiently small total  variation which intersects the first function, does so in uncountably many points.

In this article it is shown that: 
\begin{itemize}

 \item Every function of bounded variation has permeable graph (Theorem \ref{thm:bdd}).
 \item Typical (in the sense of Baire category) continuous functions have permeable graphs (Theorem \ref{thm:typical}).
 \item There exists a Hölder continuous function with an impermeable graph (Theorem \ref{thm:hoelder-ex}).
 \item There are examples of continuous functions $g$ with impermeable graphs, which are `extreme' 
 in that their points of intersection with every $f$ with $V_0^1(f)<1$, $f(0)=0$ 
 has Hausdorff dimension 1 (Theorem \ref{*impHau}).
\end{itemize}
Using Theorem \ref{thm:bdd}  we further learn:

\begin{itemize}
 \item Every  Lipschitz function has a permeable graph. In particular, every $C^1$ function has a permeable graph. 
 \item Every absolutely continuous function has a permeable graph.
 \end{itemize}
Thus, we give an incomplete but still very informative first  study of the  permeable graph property. Another result shows that
permeable functions can still be rather wild: 
\begin{theorem*}[{\cite[Theorem 8]{Brown1999}}]
If $M$ and $N$ are first category subsets of $[0, 1]$, then there exists a function $g\colon [0,1]\to\R$ with range $[0,1]$ and such that for all $m\in \R$, $t\mapsto g(t)+mt$ is not monotone  on any subinterval of $[0,1]$ and
\begin{enumerate}
\item for every $y \in  N$, $g^{-1}(y)$ is finite and
\item for every $y \in \R$, $g^{-1}(y)\cap M$ is finite.
\end{enumerate}
\end{theorem*}
Thus choosing $N=\Q\cap[0,1]$ (or another dense set of first category), and considering  the function 
$g$ given by the cited theorem, we can, for each $\delta>0$, construct a function $f$ with $V_0^1(f)<\delta$ meeting $g$ in at most finitely many points in the same way as in the proof of Theorem \ref{thm:bdd}. Therefore the graph of $g$ is permeable.\medskip

On the other hand, the square root of the modulus of the so-called Volterra-function, see, e.g.~\cite{PonceCampuzano2015VitoVC}, provides us with a function of unbounded total variation on every
open interval that intersects the Smith-Volterra-Cantor-set (a.k.a.~fat Cantor set), which has Lebesgue measure 1/2.
Nevertheless, this function has a finitely permeable graph.\bigskip

In Section \ref{sec:metric-top} we apply the earlier results to a question from metric topology.
First we recall  the permeability of sets, cf.~Definition 
\ref{def:permeable}, which was first introduced in 
 \cite{leoste21}, where the following `removability-theorem' was  also  proven
  (the definition of intrinsically $L$-Lipschitz continuous functions will be  given  in Definition \ref{def:pw-lip}):
 
\begin{theorem*}[{\cite[Theorem 15]{leoste21}}]
Let $M,Y$ be metric spaces and let $\Theta\subseteq M$ be permeable. Then every continuous function $f\colon M\to \msr$,
which is intrinsically $L$-Lipschitz continuous on $E=M\setminus \Theta$, is intrinsically $L$-Lipschitz continuous on the whole of $M$.
\end{theorem*} 

A priori it is not clear why one needs a condition on $\Theta$ besides having empty interior.
 It turns out, however, that in the case where $M=\R$,
the reverse conclusion in the above theorem, cf.~\cite[Theorem 23]{leoste21} also holds.

In the case of  $M=\R^2$ there is an example (see \cite[Proposition 26]{leoste21}) of an impermeable subset 
$\Theta$ such that every continuous function $f$ on $\R^2$ which is intrinsically Lipschitz continuous on $\R^2\setminus \Theta$
is also Lipschitz continuous on $\R^2$, but with a different Lipschitz constant.  

This example is not a manifold and not a closed subset of $\R^2$.  It turns out that it is considerably 
 more difficult 
 to construct counterexamples $\Theta$ that are closed (as subsets of $\R^d$) topological submanifolds. 
It is shown in \cite{leoste21}, that closed Lipschitz topological submanifolds 
are permeable.
In Theorem \ref{*thmintr}, using the function constructed
in Theorem \ref{thm:hoelder-ex}, we construct an example of an impermeable closed topological (even Hölder) submanifold of $\R^2$,
for which the conclusion of \cite[Theorem 15]{leoste21} does not hold.\\

The original motivation for the theory laid out in this article comes from the study of stochastic differential
equations (SDEs) whose coefficient functions are discontinuous along a manifold, but differentiable with bounded derivative elsewhere. 
Such SDEs are studied in \cite{sz2016b}, where a transform of the state space is devised so that the drift coefficient of the transformed SDE is continuous and intrinsically Lipschitz continuous. 
Using \cite[Lemma 3.6.]{sz2016b}, which is a special case of \cite[Theorem 15]{leoste21}, one can conclude that the drift of the transformed equation is Lipschitz continuous. This makes it possible to use classical results for proving existence and uniqueness of a solution, as well as convergence of numerical methods.
This ``transformation method'' from \cite{sz2016b} has widely been used since, see, e.g., \cite{sz2017c,sz2019,tmgly20,tmgly22}, and our  study here provides insights into the limitations and possible generalizations of that method applied to SDEs or related PDEs. \\

We conclude the introduction with an outline of the paper. 
Section \ref{sec:main-example} contains the `real functions'  results, where we first recall some definitions and notations in Subsection \ref{sec:notation} and show in Subsection \ref{sec:typical} that 
the set of continuous functions with permeable paths contains a dense $G_\delta$-set.
Then, in Subsection \ref{sec:hoelder-example}, we construct our  Hölder continuous example function with impermeable graph 
using nested sequences of chains of rectangles.
In Subsection \ref{*secHau}, we modify the earlier construction to obtain a still continuous function (but not Hölder continuous)
such that its points of intersection with every $f$ satisfying  $V(f)<1$ and $f(0)=0$ 
 has Hausdorff dimension 1.
 
 Section \ref{sec:metric-top} contains the application of the earlier results in metric topology.
 First  Subsection \ref{sec:pwlm} recalls the notions of intrinsic metric and intrinsically Lipschitz continuous functions in 
 general metric spaces as well as the concept of a permeable subset of a metric space, as originally introduced in 
 \cite{leoste21}. 
 In Subsection \ref{sec:graphs-paths} we show that a continuous real function has a permeable graph if and only if its 
 graph  is a permeable subset of $\R^2$. Finally, in Subsection \ref{sec:cantor-madness}, we present the construction of a continuous function
 $f\colon \R^2\to \R$ which is intrinsically Lipschitz continuous on $\R^2\setminus \Theta$, where $\Theta$ 
 is the graph of a Hölder continuous function $\theta\colon\R\to\R$, but  $f$  is not Lipschitz continuous 
 with respect to the Euclidean metric on $\R^2$.

\section{Main Examples}\label{sec:main-example}

\subsection{Definitions and notation}\label{sec:notation}
Before we start the construction of our example, we recall some concepts and notation:
Let $(M,d)$ denote a general metric space.
\begin{itemize}
\item For $x\in M$ and $r>0$, let $$B_r(x):=\left\{y\in M: d(x,y)<r\right\}$$
denote the open ball with radius $r$ and center $x$ in $M$.
\item For a subset $A\subseteq M$  we denote by $\overline A$ the topological closure and by $\partial A$ the topological boundary  of $A$. 

\item For a set $U\subseteq M$, let $|U|:=\sup\{d(x,y):x,y\in M\}$ denote  the  diameter of the set $U$.
\item 
Given $s>0$,  the  $s$-dimensional {\em Hausdorff measure} of the set $A\subseteq M$  is defined by
\[ \mathscr{H}^s(A):=\lim_{\varepsilon\to 0}\inf \Big\{\sum_{i=1}^\infty |U_i|^s: \forall i\geq 1:|U_i|<\varepsilon \text{ and }A\subseteq \bigcup_{i=1}^\infty U_i \Big\}. \] 

\item Let 

\begin{align*}
\hdim (A):&=\inf\big\{s:\mathscr{H}^s(A)=0\big\}\\
&=\inf\Big\{s:\exists C_s>0: \forall \varepsilon>0: \exists (U_i)_{i\geq 1}: \forall i\geq 1: |U_i|<\varepsilon, \\
&\quad\quad\quad\quad\quad\quad\quad\quad\quad\quad\quad\quad A\subseteq \bigcup_{i=1}^\infty U_i,\text{ and }\sum_{i=1}^\infty|U_i|^s<C_s\Big\},
\end{align*}

denote the {\em Hausdorff-dimension} of $A$.
\item For a function $f\colon [a,b]\to \R$ let $V_a^b(f)\in [0,\infty]$ denote the total variation of $f$ over the interval $[a,b]$, 
i.e. 
\[
V_a^b(f)=\sup\Big \{\sum_{k=1}^n |f(x_k)-f(x_{k-1})| : n\in \mathbb{N},a\le x_0<\ldots<x_n\le b  \Big \}\,.
\]  

\item For a function $f\colon [a,b]\to \R$ let $\ell_a^b(f)\in[0,\infty]$ denote the length of the graph of $f$ over the interval $[a,b]$, i.e. 
\begin{align*}
\ell_a^b(f)=\sup\Big\{\sum_{k=1}^n &\sqrt{|x_k-x_{k-1}|^2+ |f(x_k)-f(x_{k-1})|^2} \colon\\
& \qquad\qquad n\in \mathbb{N},a\le x_0<\ldots<x_n\le b  \Big \}\,.
\end{align*}

\item For $m\in \N$ and $1\leq i\leq m$, let $\mathrm{pr}_i\colon\R^m\to\R$ denote the projection onto the $i$th component.
\item For a mapping $\gamma\colon [a,b]\to \R^2$ let $\gamma_1:=\pr_1\circ\gamma,\gamma_2:=\pr_2\circ\gamma$ be the coordinate functions.
\item For a subset $A\subseteq \R^d$ we write $\conv(A)$ for the convex hull of $A$. 
\end{itemize}

\subsection{Typical continuous functions have permeable graphs}\label{sec:typical}

In this subsection we show that a {\em typical} continuous function $g\in C[0,1]$ has a finitely permeable  graph. ``Typical'' here means that the set of functions with finitely permeable graph contains a dense $G_\delta$ set. 
On $C[0,1]$ we use the topology induced by the $\sup$-norm, $||f||=\sup_{x\in[0,1]}|f(x)|$. 
The precise formulation is given in the next theorem.

\begin{theorem}\label{thm:typical}
There is a dense $G_\delta$ set $\mathcal{G}\subseteq C[0,1]$ such that for all $g\in \mathcal{G}$, for all $y\in \R$ and all $\delta>0$ there is a function $f\colon[0,1]\to\mathbb{R}$ with finitely many jumps such that $f(0)=y, V_0^1(f)<\delta$ and $\{t:f(t)=g(t)\}\subseteq \{0\}$.
\end{theorem}
\begin{proof}

Let $(g_n)_{n\geq 1}$ be an enumeration of a dense set $D$ of piecewise linear functions which are nonconstant on any interval. For all $n$  let $v_n$ be the minimal number of  intervals on which $g_n$ is linear and which partition $[0,1]$. Then define 
 the numbers $\rho_{n,k}:=\frac{1}{2k(v_n+1)}$ and the sets
\[\mathcal{G}_k:=\bigcup_{n=0}^\infty B_{\rho_{n,k}}(g_n),\quad \mathcal{G}=\bigcap_{k=1}^\infty \mathcal{G}_k.\]
Clearly, the sets
 $\mathcal{G}_k$ are open and $\mathcal{G}$ is a dense $G_\delta$ set in $C[0,1]$. Let $g\in \mathcal{G}$ and $\delta>0$ be given. Choose $k>0$ such that $\frac{1}{k}<\delta$. Then there is some $n_k$ and a function $g_{n_k}\in D$ such that $g\in B_{\rho_{n_k,k}}(g_{n_k})$. 

Take now $y\in \R$ arbitrary. We will  construct a function $f$ with $V_0^1(f)<\delta$ and $\{t:f(t)=g(t)\}\subseteq \{0\}$.

\begin{figure}
\begin{center}
\begingroup%
  \makeatletter%
  \providecommand\color[2][]{%
    \errmessage{(Inkscape) Color is used for the text in Inkscape, but the package 'color.sty' is not loaded}%
    \renewcommand\color[2][]{}%
  }%
  \providecommand\transparent[1]{%
    \errmessage{(Inkscape) Transparency is used (non-zero) for the text in Inkscape, but the package 'transparent.sty' is not loaded}%
    \renewcommand\transparent[1]{}%
  }%
  \providecommand\rotatebox[2]{#2}%
  \newcommand*\fsize{\dimexpr\f@size pt\relax}%
  \newcommand*\lineheight[1]{\fontsize{\fsize}{#1\fsize}\selectfont}%
  \ifx\svgwidth\undefined%
    \setlength{\unitlength}{331.37678215bp}%
    \ifx\svgscale\undefined%
      \relax%
    \else%
      \setlength{\unitlength}{\unitlength * \real{\svgscale}}%
    \fi%
  \else%
    \setlength{\unitlength}{\svgwidth}%
  \fi%
  \global\let\svgwidth\undefined%
  \global\let\svgscale\undefined%
  \makeatother%
  \begin{picture}(1,0.60963235)%
    \lineheight{1}%
    \setlength\tabcolsep{0pt}%
    \put(0,0){\includegraphics[width=\unitlength,page=1]{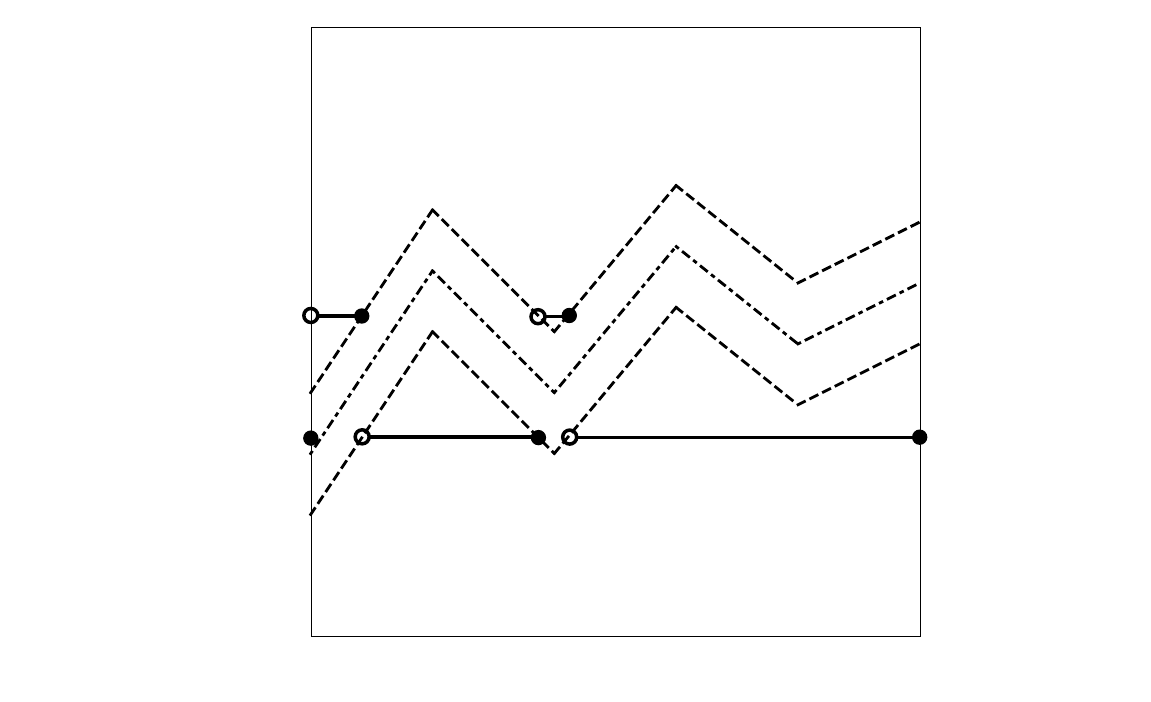}}%
    \put(0.25637832,0.20618146){\makebox(0,0)[rt]{\lineheight{1.25}\smash{\begin{tabular}[t]{r}$y$\end{tabular}}}}%
    \put(0.26004787,0.32988554){\makebox(0,0)[rt]{\lineheight{1.25}\smash{\begin{tabular}[t]{r}$y+2\rho_{n_k,k}$\end{tabular}}}}%
    \put(0,0){\includegraphics[width=\unitlength,page=2]{typical-neu.pdf}}%
    \put(0.27346434,0.02180333){\makebox(0,0)[t]{\lineheight{1.25}\smash{\begin{tabular}[t]{c}$0$\end{tabular}}}}%
    \put(0.80160533,0.01899458){\makebox(0,0)[t]{\lineheight{1.25}\smash{\begin{tabular}[t]{c}$1$\end{tabular}}}}%
    \put(0.26178286,0.26135948){\makebox(0,0)[rt]{\lineheight{1.25}\smash{\begin{tabular}[t]{r}$y+\rho_{n_k,k}$\end{tabular}}}}%
    \put(0.80948952,0.35305147){\makebox(0,0)[lt]{\lineheight{1.25}\smash{\begin{tabular}[t]{l}$g_{n_k}$\end{tabular}}}}%
    \put(0,0){\includegraphics[width=\unitlength,page=3]{typical-neu.pdf}}%
    \put(0.33728494,0.02160147){\makebox(0,0)[t]{\lineheight{1.25}\smash{\begin{tabular}[t]{c}$t_{j-1}$\end{tabular}}}}%
    \put(0.47439949,0.02221425){\makebox(0,0)[t]{\lineheight{1.25}\smash{\begin{tabular}[t]{c}$t_{j}$\end{tabular}}}}%
    \put(0,0){\includegraphics[width=\unitlength,page=4]{typical-neu.pdf}}%
  \end{picture}%
\endgroup%
\end{center}
\label{fig1}
\caption{An open ball of radius $\rho_{n_{k},k}$ around a piecewise linear continuous function $g_{n_{k}}$ together 
with a piecewise constant function which avoids the ball.}
\end{figure}

This intersection set is nonempty only in the case when $g(0)=y$. Let $t_0:=0$, 
$$f_0 :=\begin{cases}
y, & \text{ if }y\notin[g_{n_k}(0)-\rho_{n_k,k},g_{n_k}(0)+\rho_{n_k,k}], \\
y+2\rho_{n_k,k}, & \text{ if }y\in[g_{n_k}(0),g_{n_k}(0)+\rho_{n_k,k}] ,\\
y-2\rho_{n_k,k}, & \text{ if }y\in[g_{n_k}(0)-\rho_{n_k,k},g_{n_k}(0)),
\end{cases}$$ 
and, inductively, for $j=1,...,v_{n_{k}}+1$   define 
$$t_j:=\begin{cases}{\renewcommand{\arraystretch}{0.6}\scriptsize \begin{array}{l}\inf\{t>t_{j-1}:f_{j-1}=g_{n_k}(t)+\rho_{n_k,k}
\\ \quad\quad\quad\quad\;\;\, \text{ or } f_{j-1}=g_{n_k}(t)-\rho_{n_k,k}\},\end{array}}
\quad\quad&
\text{if defined and }
t_{j-1}\neq 1,
\\
\quad 1, &\text{otherwise}.
\end{cases}$$
Observe  that $t_{v_{n_{k}}+1}=1$ and it may happen that 
$t_{j}=1$ for $j$ less than $v_{n_{k}}+1$.
For all $j$ with $t_{j}<1$ we also put
$$f_j:=\begin{cases}
f_{j-1}+2\rho_{n_k,k}, & \text{ if } g_{n_k} \text{ decreases at }t_j, \\
f_{j-1}-2\rho_{n_k,k}, &\text{ if } g_{n_k} \text{ increases at }t_j, \\
f_{j-1}, & \text{ if } g_{n_k} \text{ has a local extremum at }t_j.
\end{cases}$$
Assume that $0<t_1<\dotsb<t_{v_{n_k}}<t_{v_{n_k}+1}=1$, all other cases are similar but easier to treat. We define $f$ so that it is piecewise constant on the intervals $]t_{j-1},t_j]$, $j=1,\dotsc,v_{n_k}+1 $ with
values $f(t)=f_{j-1}$ and $f(0)=y$.

As a locally constant function, $f$ obtains its variation only by means of its jumps, which may occur 
at $t_0,\dotsc,t_{v_{n_k}}$ and have jump heights at most $2\rho_{n_k,k}$. Hence,
$V_0^1(f)\leq \sum_{j=0}^{v_{n_k}}2\rho_{n_k,k}= 2(v_{n_k}+1)\rho_{n_k,k}=\frac{1}{k}<\delta$ by the definition of $\rho_{n_k,k}$.
\end{proof}

\begin{corollary}
There is a dense $G_\delta$ set $\mathcal{G}\subseteq C[0,1]$ such that for all $g\in \mathcal{G}$, $g$ has a permeable graph.
\end{corollary}

\begin{proof}
Let $\mathcal{G}$ be as in the proof of Theorem \ref{thm:typical}  and  let 
$g\in \mathcal{G}$. Let $y\in \R$ and $\delta>0$. Then there exists
$\tilde f\colon [0,1]\to \R$ with $\tilde f(0)=y$, and $V_0^1(\tilde f)<\frac{\delta}{2}$ so that 
$\{t:\tilde f(t)=g(t)\}\subseteq \{0\}$. Define
\[
f(t):=\begin{cases}
\tilde f(t), & \text{ if }t<1,\\
y, & \text{ if }t=1.
\end{cases}
\]
Then $V_0^1(f)\le V_0^1(\tilde f)+|\tilde f(1)-y|\le 2V_0^1(\tilde f)< \delta$
and $\{t:\tilde f(t)=g(t)\}\subseteq \{0,1\}$.
\end{proof}

\subsection{Continuous functions with impermeable graph}\label{sec:hoelder-example}

In this subsection we are going to construct a function $g_H\colon [0,1]\to \R$ with an impermeable graph. This function will be Hölder continuous.
In Section \ref{*secHau} we will modify this  construction, which will yield 
a  continuous, non Hölder continuous function $g_C$, and  with the property that the intersection set with every $f$ satisfying $f(0)=0$ and $V_0^1(f)<1$ has Hausdorff dimension equal to 1.  At a heuristic level it is clear that the price we need to pay to obtain larger intersection sets is that our function $g_C$ should become wilder.

Both functions $g_H$ and $g_C$ will be constructed as a limit of a nested sequence of chains of rectangles, 
a concept that we are going to  define now.

\begin{definition}\label{def:nested-collection}
\phantom{a}
\begin{enumerate}
\item Let $(R_{n})_{n=1}^N$ be a finite collection  of axes-parallel rectangles contained in $[0,1]\times \R$, 
 with common width $w$ and height $h$. Let $(x_n,y_n)$ be the lower left corner of $R_n=[x_{n},x_{n}+w]\times [y_{n},y_{n}+h]$.
 We call $(R_{n})_{n=1}^N$ a {\em  connected chain of rectangles}, if 
 \begin{enumerate}
 \item $x_1=0$, $x_N+w=1$;
 \item $x_{n+1}=x_{n}+w$ for $n=1,...,N-1$;
 \item $|y_{n+1}-y_{n}|\le h$  for $n=1,...,N-1$.
\end{enumerate}
\item\label{it:nested-collection-sequence} Let $\{R_{n,k}\colon k\in \N_0, n=1,\ldots,N_k\}$ be a finite collection  of  rectangles contained in $[0,1]\times \R$. 
Let $(x_{n,k},y_{n,k})$ be the lower left corner of $R_{n,k}=[x_{n,k},x_{n,k}+w_k]\times [y_{n,k},y_{n,k}+h_k]$.
We call $(R_{n,k})$ {\em a nested sequence of chains of rectangles}, if 
 \begin{enumerate}
\item for every $k$, $(R_{n,k})_{n=1}^{N_k}$ is a connected chain of rectangles  with corresponding widths $w_k$ and heights $h_k$; 
\item $N_{k+1}=d_{k+1} N_k$ for some $d_{k+1}\in\N\setminus\{1\}$;
\item\label{it:nested-collection-sequence-3} $y_{j,k}\le y_{(j-1) d_{k+1}+n,k+1}\le y_{j,k}+h_k-h_{k+1}$ for all $j\in \{1,\ldots,{N_k}\}$, $n\in 1,\ldots,d_{k+1}$, $k\in\N_0$.
\end{enumerate}
\end{enumerate}
See Figures \ref{*fignested1} and \ref{*figrectangles}.
\end{definition}
Note that, in a  nested sequence of chains of rectangles $\{R_{n,k}\colon k\in \N_0, n=1,\ldots,N_k\}$,
every rectangle $R_{n,k+1}$ is  contained in precisely one rectangle $R_{m,k}$.

\begin{figure}[h]
\begin{center}
\begingroup%
  \makeatletter%
  \providecommand\color[2][]{%
    \errmessage{(Inkscape) Color is used for the text in Inkscape, but the package 'color.sty' is not loaded}%
    \renewcommand\color[2][]{}%
  }%
  \providecommand\transparent[1]{%
    \errmessage{(Inkscape) Transparency is used (non-zero) for the text in Inkscape, but the package 'transparent.sty' is not loaded}%
    \renewcommand\transparent[1]{}%
  }%
  \providecommand\rotatebox[2]{#2}%
  \newcommand*\fsize{\dimexpr\f@size pt\relax}%
  \newcommand*\lineheight[1]{\fontsize{\fsize}{#1\fsize}\selectfont}%
  \ifx\svgwidth\undefined%
    \setlength{\unitlength}{223.89598871bp}%
    \ifx\svgscale\undefined%
      \relax%
    \else%
      \setlength{\unitlength}{\unitlength * \real{\svgscale}}%
    \fi%
  \else%
    \setlength{\unitlength}{\svgwidth}%
  \fi%
  \global\let\svgwidth\undefined%
  \global\let\svgscale\undefined%
  \makeatother%
  \begin{picture}(1,0.74885615)%
    \lineheight{1}%
    \setlength\tabcolsep{0pt}%
    \put(0,0){\includegraphics[width=\unitlength,page=1]{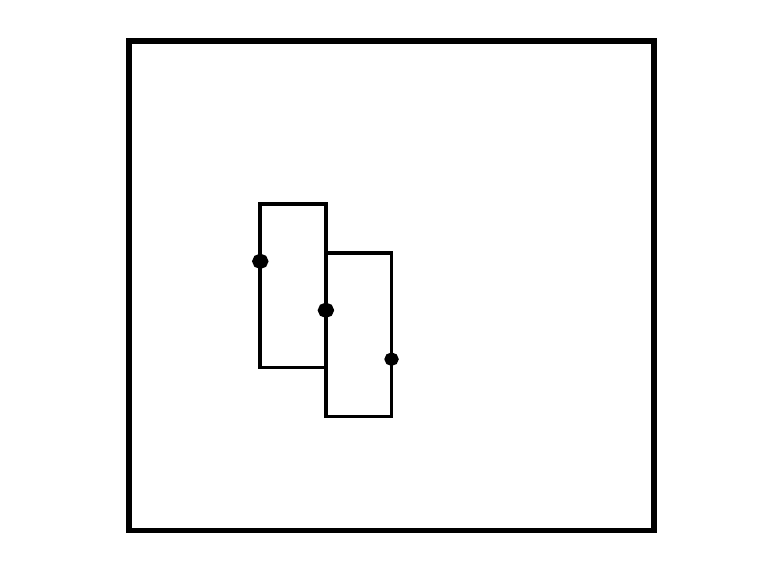}}%
    \put(0.2760027,0.40483035){\makebox(0,0)[t]{\lineheight{1.25}\smash{\begin{tabular}[t]{c}$a_{n,k}$\end{tabular}}}}%
    \put(0.57578734,0.27892943){\makebox(0,0)[t]{\lineheight{1.25}\smash{\begin{tabular}[t]{c}$b_{n+1,k}$\end{tabular}}}}%
    \put(0.55530569,0.35864455){\makebox(0,0)[t]{\lineheight{1.25}\smash{\begin{tabular}[t]{c}$b_{n,k}=a_{n+1,k}$\end{tabular}}}}%
    \put(0,0){\includegraphics[width=\unitlength,page=2]{nested1-note.pdf}}%
    \put(0.14293165,0.26961065){\makebox(0,0)[rt]{\lineheight{1.25}\smash{\begin{tabular}[t]{r}$y_{n,k}$\end{tabular}}}}%
    \put(0.14251992,0.47963934){\makebox(0,0)[rt]{\lineheight{1.25}\smash{\begin{tabular}[t]{r}$y_{n,k}+h_k$\end{tabular}}}}%
    \put(0.3540114,0.02320473){\makebox(0,0)[t]{\lineheight{1.25}\smash{\begin{tabular}[t]{c}$x_{n,k}$\end{tabular}}}}%
    \put(0.41835049,0.02359899){\makebox(0,0)[lt]{\lineheight{1.25}\smash{\begin{tabular}[t]{l}$x_{n,k}+w_k$\end{tabular}}}}%
    \put(0,0){\includegraphics[width=\unitlength,page=3]{nested1-note.pdf}}%
  \end{picture}%
\endgroup%
\end{center}
\caption{Two rectangles in a chain of rectangles $(R_{n,k})$.}\label{*fignested1}
\end{figure}

Nested sequences of chains of rectangles can be used to represent continuous functions:

\begin{proposition}\label{lem:g-rect}
 \begin{enumerate}
 \item\label{it:prop1-expansion} Let $\{R_{n,k}\colon k\in \N_0, n=1,\ldots,N_k\}$ be a nested sequence of chains of rectangles.
 Then for every $t\in [0,1]$ one can select a sequence $(m(k,t))_{k=0}^\infty$ such that 
 $t\in [m(k,t)w_k,(m(k,t)+1)w_k]$ for all $k\in \N_0$. The sequence is unique for all but countably 
 many $t$.
\item \label{it:prop1-g-to-rect}If $g\colon [0,1]\to \R$ is a continuous function, and $(N_k)_{k=1}^\infty$ is a sequence of natural numbers with $N_{k+1}/N_k\in \N\setminus \{1\}$ then there exists a  nested sequence of chains of rectangles $\{R_{n,k}\colon k\in \N_0, n=1,\ldots,N_k\}$ with corresponding sequence of heights 
$(h_k)_{k=0}^\infty$, such that $\lim_{k\to \infty}h_k=0$ and 
\[
g(t)=\lim_{k\to \infty} y_{m(k,t),k} 
\]
for all $t\in [0,1]$. Moreover, $y_{m(k,t),k}\le  g(t) \le y_{m(k,t),k}+h_k$ for every $t\in [0,1]$ and $k\in \N_0$. 
\item \label{it:prop1-rect-to-g}If $\{R_{n,k}\colon k\in \N_0, n=1,\ldots,N_k\}$ is a  nested sequence of connected chains of rectangles  with a corresponding sequence of heights $(h_k)_{k=0}^\infty$ such that $\lim_{k\to \infty}h_k=0$, 
then 
\begin{equation}\label{*gteq}
g(t):=\lim_{k\to \infty} y_{m(k,t),k} 
\end{equation}
 defines a continuous function $g\colon [0,1]\to \R$ and
$y_{m(k,t),k}\le  g(t) \le y_{m(k,t),k}+h_k$ for every $t\in [0,1]$ and $k\in \N_0$, and
\[
\operatorname{graph}(g)=\bigcap_{k=0}^\infty\bigcup_{n=1}^{N_k} R_{n,k}\,.
\]
\end{enumerate}
\end{proposition}

\begin{proof} The proof is left to the reader.
\end{proof}

\begin{figure}[h]
\begin{center}
\includegraphics[width=3cm]{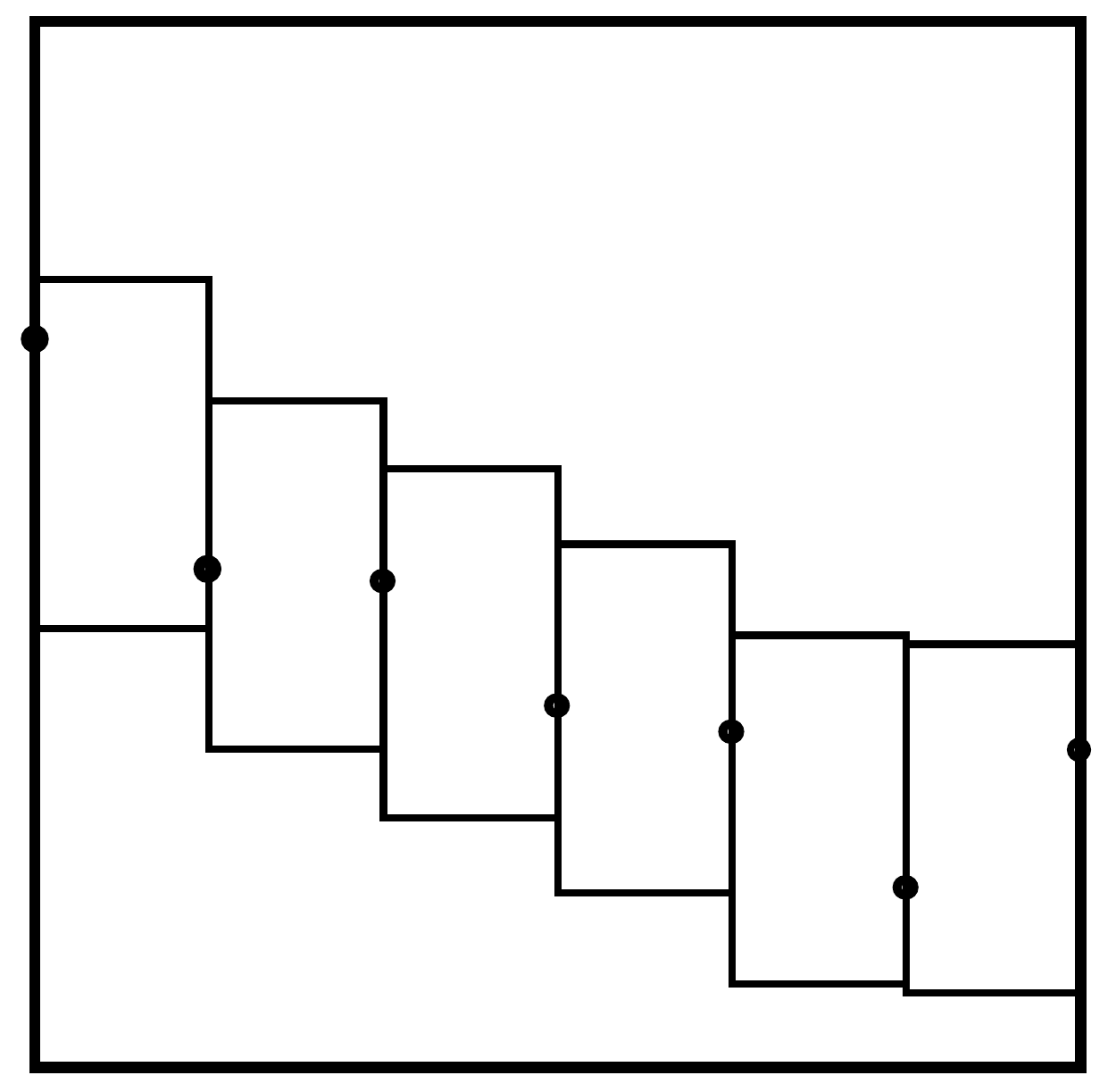}
\includegraphics[width=3cm]{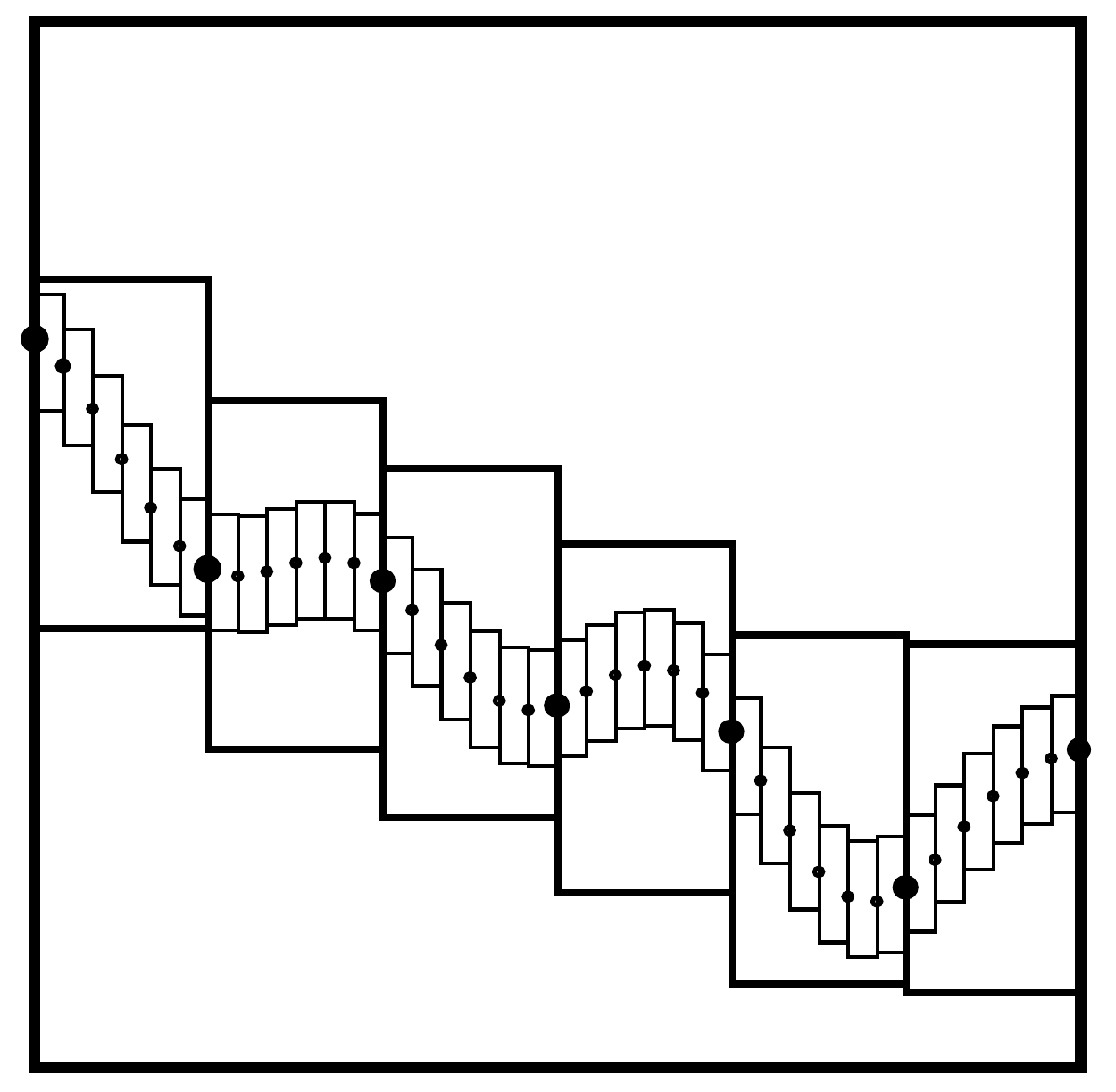}
\includegraphics[width=3cm]{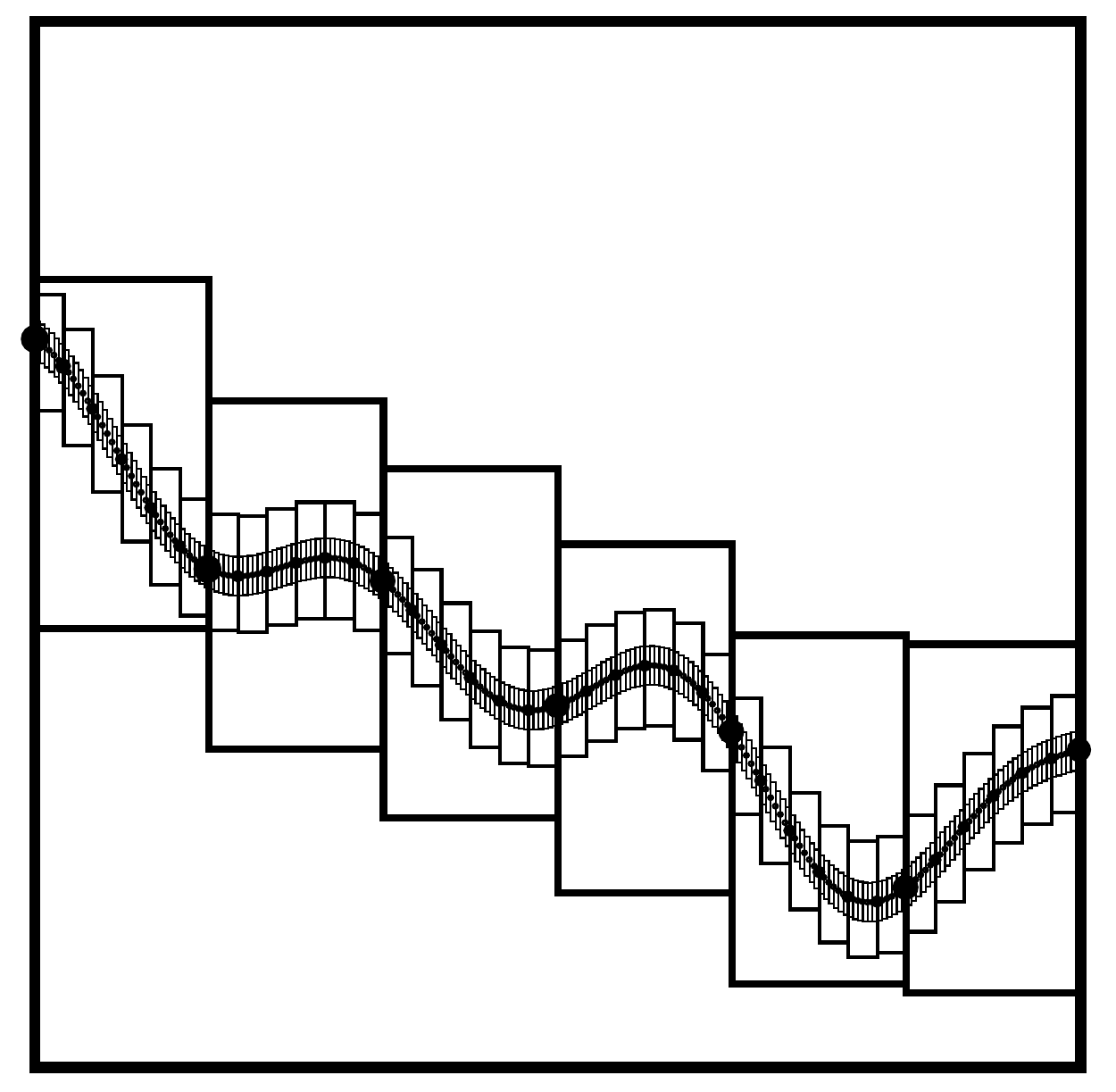}
\end{center}
\caption{Nested rectangles on different levels.}\label{*figrectangles}
\end{figure}

\medskip

\begin{lemma}\label{lem:rectangle-to-hoelder}
 Let $(R_{n,k})_{n,k}$ be a nested sequence of chains of rectangles.
 Assume further that there exist $j \in \N \setminus {\{ 1 \}}$ and $\beta\in (0,1)$  with $w_k=j^{-k}$ and $h_k=\beta^k$.
 Then $g$ as defined in \eqref{*gteq} is Hölder-continuous with exponent 
 $\alpha:=$$-\frac{\log(\beta)}{\log(j)}$.
\end{lemma}

\begin{proof}
By continuity of $g$ it is enough to show the Hölder property on the dense subset 
$\{nj^{-k}\colon k\in \N, 0\le n\le j^k\}$. First we note that
\begin{equation}\label{eq:hoelder1}
|g((n+1)j^{-k})-g(nj^{-k})|\le h_k =\beta^k =j^{k\frac{\log(\beta)}{\log(j)}}=(j^{-k})^{\alpha}\,.
\end{equation}
Now let $x,y\in\{nj^{-k}\colon k\in \N, 0\le n\le j^k\}$, $x\not=y$.  Consider the $j$-adic representations 
\[
x= \sum_{k=1}^K x_k j^{-k}\,,\quad y= \sum_{k=1}^K y_k j^{-k}\,,
\]
with $x_k,y_k\in \{0,\ldots,j-1\}$ for all $k$.
Choose $k_{0}$ such that the following inequality is satisfied
\begin{equation}\label{*eqxy}
j^{-k_0}<|x-y|\le j^{-k_0+1}\,.
\end{equation}
Let $x^{*}=\sum_{\ell=0}^{k_{0}} x_\ell j^{-\ell}$ and $y^{*}=\sum_{\ell=0}^{k_{0}} y_\ell j^{-\ell}$. Then by \eqref{eq:hoelder1}
\begin{equation}\label{*xcsi}
\max \{|g(x)-g(x^{*})|, |g(y)-g(y^{*})|\} \leq (j-1)\sum_{\ell =k_{0}+1}^{K}j^{-\ell \alpha} \leq \frac{(j-1)j^{-(k_{0}+1)\alpha}}{1-j^{-\alpha}}.
\end{equation}
Also by \eqref{eq:hoelder1} we have
$$|g(x^*)-g(y^*)|\leq j\cdot j^{-k_{0}\alpha}.$$
By using \eqref{*eqxy}, \eqref{*xcsi} and the triangle inequality we obtain that
$$ 
|g(x)-g(y)|<\Big (2\frac{(j-1)j^{-\alpha}}{1-j^{-\alpha}} +j \Big ) j^{-k_{0}\alpha}<
\Big ( 2\frac{(j-1)j^{-\alpha}}{1-j^{-\alpha}} +j \Big )|x-y|^{\alpha}.
$$
This finishes the proof.
\end{proof}

In the following we construct a concrete example of a Hölder-continuous function $g_H$. 
From now on  we will always assume $h_k:=\frac{10}{10^k}$ for all $k\in \N_0$.

We assume that each rectangle $R_{n,k}$ comes with an {\em entry point} 
$a_{n,k}\in [y_{n,k},y_{n,k}+h_k]$ and an {\em exit point} $b_{n,k}\in [y_{n,k},y_{n,k}+h_k]$
such that $a_{n+1,k}=b_{n,k}$ for $n=1,\ldots,N_k-1$. It is easy to see that the  $R_{n,k}$,
$n\in\{1,\ldots,N_{k}\}$
 form a connected chain iff such entry/exit points exist and conversely, given $w_k,h_k$, the 
entry/exit points determine the chain. 

Finally, we require $a_{m d_k+1,k+1}=a_{m+1,k}$ and $b_{(m+1) d_k,k+1}=b_{m+1,k}$ for all $m
=0,\ldots,N_{k}-1$.

We start with $R_{1,0}:=[0,1]\times[-5,5]$, such that $(x_{1,0},y_{1,0})=(0,-5)$, $w_0=1$, $h_0=10$.
Further, we set $a_{1,0}:=\frac{5}{2}$, $b_{1,0}:=-\frac{5}{2}$. 

We build a nested sequence of
 connected
 chains of rectangles $\{R_{n,k}\colon k\in \N_0, 1\le n \le N_k\}$  from two kinds of rectangles: ``ascending'' and ``descending'' ones: 
\begin{convention*}
We call a rectangle $R_{n,k}$ with entry point $a_{n,k}$ and exit point $b_{n,k}$
\begin{enumerate}
\item {\em ascending} if $a_{n,k}=y_{n,k}+\frac{1}{4}h_k$ and 
$b_{n,k}=y_{n,k}+\frac{3}{4}h_k$
\item {\em descending} if $a_{n,k}=y_{n,k}+\frac{3}{4}h_k$ and 
$b_{n,k}=y_{n,k}+\frac{1}{4}h_k$\,.
\end{enumerate}
\end{convention*}

\medskip

Thus, in the sense of this convention,  $R_{1,0}$ is descending.

Now in level $k=1$, start with connecting 14 descending rectangles, such that  $b_{14,1}=
\frac{5}{2}-14\cdot\frac{5}{10}=-\frac{45}{10}$ (recall that $h_1=1$). 
Continue by  connecting 18 ascending rectangles, 
such that  $b_{14+18,1}=-\frac{45}{10}+18\cdot\frac{5}{10}=\frac{45}{10}$. 
We proceed with connecting another 11   groups of each 18 descending and ascending rectangles  such that 
the interval $[-4,4]$ is crossed 12  times. We end with 4 ascending rectangles such that 
$b_{N_1,1}=-\frac{5}{2}$. (So $N_1=14+ 12 \cdot 18+4=234 =d_1$, $\frac{1}{234} =w_1$).
This procedure is repeated for every $k$, where for the descending rectangles the nesting is precisely
as for $k=1$, such that in each rectangle $R_{n,{k-1}}$ we cross $12 $ times the interval $[y_{n,{k-1}}+\frac{1}{10^{k}},y_{n,{k-1}}+\frac{1}{10^{k-1}}-\frac{1}{10^{k}}]$, yielding $14+12\cdot 18+4 $ rectangle descendants with width $w_k=234^{-k} $ for $k\geq 1$. For ascending rectangles we start with 14 ascending rectangles and continue in
the analogous way. 

Together with $h_k=10^{1-k}$, the construction of our function $g_H$  is thus finished 
via Proposition  \ref{lem:g-rect}, Item \ref{it:prop1-rect-to-g}, and we need to show that $g_H$  
has an impermeable graph .

For that purpose, take an arbitrary function $f\colon [0,1]\to \R$ with $f(0)\le \frac{5}{2}$ and 
$f(1)\ge -\frac{5}{2}$ and 
$V_0^1(f)<1$. We will show that $\{t\in [0,1]\colon f(t)=g_H(t)\}$ is uncountable. 

Since $V_0^1(f)<1$, the graph of $f$ is contained in $[0,1]\times [-\frac{7}{2},\frac{7}{2}]$.

Since the chain $R_{n,1}$, $n=1,\ldots,N_1$, alternates $12$  times between $-\frac{45}{10}$ and 
$\frac{45}{10}$  it must cross  the interval $[-\frac{7}{2},\frac{7}{2}]$
at least that many times, and so must the graph of $g_H$ .

We say the graph of $f$ is {\em upcrossing} a decreasing rectangle  
$R_{n,k}$, if  $f(x_{n,k})\le g_H(x_{n,k})=a_{n,k}$  
and $f(x_{n,k}+w_k)\ge g_H(x_{n,k}+w_k)=b_{n,k}$. 
We say an upcrossing is {\em good}, if 
$V_{x_{n,k}}^{x_{n,k}+w_k}(f)<10^{-k}$. In an analogous way we may define (good) 
downcrossings of
an ascending rectangle.  A {\em good crossing} is a good up- or downcrossing. 

The  graph of $f$ is contained in the strip $[0,1]\times [-\frac{7}{2},\frac{7}{2}]$ which is crossed
12  times by the chain of rectangles $(R_{n,1})$. We must therefore have at least 6 upcrossings and 6 downcrossings. Since 
the total variation of $f$ is smaller than 1, we need to have at least 2 up- or down crossings that 
are good. Indeed, proceeding towards a contradiction, suppose that  more than 10 crossings were not good, say those at rectangles 
$R_{n_1,1},\dotsc,R_{n_{ 11},1}$. Then 
$V_0^1(f)\ge V_{x_{{n_1},1}}^{x_{{n_1},1}+w_1}(f)+\dotsb+ V_{x_{{n_{11}},1}}^{x_{{n_{11}},1}+w_1}(f)
>10\cdot \frac{1}{10}=1$, which poses a contradiction.  

For rectangles with good crossings we may repeat the argument. 
 So we get at least $12-10=2$ 
 good crossings  
 on the next level for every good crossing at the present level.
In particular there are at least $2^{\aleph_0}$-many $t\in [0,1]$ which are contained in a sequence of nested rectangles  with good crossings. For all but countably many $t\in [0,1]$ the corresponding sequence $(m(k,t))_{k\ge 0}$ is unique, so that 
$(R_{m(k,t),k})_{k\ge 0}$ is a nested sequence of rectangles with good crossings and
$x_{m(k,t),k}<t<x_{m(k,t),k}+w_k$.
Consider one such $t$. Then $f(x_{m(k,t),k})\in [y_{m(k,t),k},y_{m(k,t),k}+h_k]$,
and $g_H(t)=\lim_{k\to \infty} y_{m(k,t),k}$, 
  so if $f$ is continuous in $t$, then $f(t)=g_H(t)$. 
  
  Since $f$ is of bounded variation, $f$ has at most countably many points of discontinuity. 
  Therefore we have still $2^{\aleph_0}$-many $t\in [0,1]$ with $g_H(t)=f(t)$.

\medskip

Using Lemma \ref{lem:rectangle-to-hoelder} and the constructed example we have proven the following theorem :

\begin{theorem}\label{thm:hoelder-ex}
There exists a 
 Hölder continuous function $g_H\colon [0,1]\to\R$ with Hölder exponent $\frac{\log(10)}{\log(234)}\approx 0.42$ and with an impermeable graph.
 Specifically, for every function $f:[0,1]\to \R$ with $f(0)=0$ and $V_0^1(f)<1$ the set 
 $\{t\in [0,1]\colon g_{H}(t)=f(t)\}$ is uncountable. 
\end{theorem}

\subsection{The Hausdorff-dimension of the intersection set}\label{*secHau}

We will now modify the construction from Subsection \ref{sec:hoelder-example} 
to get a graph which is even ``more impermeable''  at the price of losing Hölder continuity. 
\begin{theorem}\label{*impHau}
There exists a continuous function $g_C\colon [0,1]\to \R$ which has impermeable graph and, in addition,   the set $\{t\in [0,1]: f(t)=g_C(t)\}$ has Hausdorff dimension $1$ for every bounded variation function $f$ for which $f(0)=0$ and $V_0^1(f)<1$. 
\end{theorem}

Before proving the theorem we recall   the {\em mass distribution principle}:
\begin{theorem*}[Mass distribution principle, {\cite[Chapter 4]{falconer}}]
Let $F\subseteq \R^d$ and let $\nu$ be a mass distribution  on $F$, i.e., a measure on $\mathcal{B}(F)$. 
Suppose that for some 
$s\geq 0$ there are numbers $c>0$ and $ {\delta}>0$ such that $  \displaystyle   {\nu}(U)\leq c|U|^{s}$
for all sets $U$ with $|U|\leq  {\delta}$. 
Then $ {\mathscr{H}}^{s}(F)\geq \nu(F)/c$ and $s\leq \hdim (F).$
\end{theorem*}


\begin{proof}[Proof of Theorem \ref{*impHau}.]
First we will construct a measure $\nu$ on $[0,1]$ having support on the set of points $t$ which are contained in a nested sequence of projections of rectangles with good crossings:

Let $g_C$ be the function defined by the above construction via a scheme of nested, connected 
chains of rectangles  as above, with $(R_{n,k})_{n=1}^{N_k}$ having height $h_k=10^{1-k}$ at level $k$  
and  $w_k=\prod_{l=1}^k (14+100^l\cdot 18+4)^{-1}$. That is, at level $k$ we have, for each rectangle
$R_{n,k-1}$, at least $100^k$ new chains of descending/ascending rectangles, crossing the strip $[x_{n,k},x_{n,k}+w_k]\times[y_{n,k}+15\cdot 10^{-k-1},y_{n,k}+10^{1-k}-15\cdot 10^{-k-1}]$  a $100^k$ times.  
By the same reasoning as in Subsection \ref{sec:hoelder-example}, 
for every good crossing at level $k-1$, we have at least $100^k-10$ good crossings at level $k$.

Let $F_{n,k}$ denote the projection   of the rectangle $R_{n,k}$  to the $x$-axis, i.e., $F_{n,k}:=[nw_k,(n+1)w_k]=[x_{n,k},x_{n,k}+w_k]$. 
Recall  that for any $t\in [0,1]$ there is a sequence $\big(m(k,t)\big)_{{k\geq 0}}$  such that $t\in F_{m(k,t),k}$ for all $k\geq 0$ and that this sequence is  unique 
except when $t\in(0,1)\cap \partial F_{n,k}$  for some $(n,k)$. For definiteness we set 
$m(k,t):=\max\{1\le n\le N_k\colon t\in F_{n,k}\}$.  This admits the following mappings (a.k.a. $F_{n,k}$-adic tree representation), 
$$
 c\colon [0,1]
\to \prod_{k\geq 0 }\{F_{n,k}:1\leq n\leq N_k\},
\qquad
t 
\mapsto \left(F_{m(k,t),k}\right)_{k\geq 0 }
$$
and
$$
c^*\colon c\left([0,1]\right),
\to [0,1],
\qquad
\left(F_{n,k}\right)_{k\geq 0 }
\mapsto\inf \bigcap_{k\geq 0 }F_{n,k}.
$$

Note that the second mapping can be extended to the domain of all nested sequences
$$\mathcal{S}:=\Big\{\varphi\in \prod_{k\geq 0 }\{F_{n,k}:1\leq n\leq N_k\}: \varphi_{k+1}\subseteq \varphi_k 
\text{ for all }k\in\N_0\Big\},$$ 
 which differs from the image of $c$ only by countably many elements.
We now endow $\prod_{k\geq 0 }\{F_{n,k}:1\leq n\leq N_k\}$ with the  product $\sigma$-algebra

$$\Sigma:=\bigotimes_{k\geq 0}\mathcal{P}\Big(\big\{F_{n,k}\colon n\in\{1,\dotsc,N_k\}\big\}\Big).$$ 
Here $\mathcal{P}(X)$ denotes the power set of a set $X$.
The function $c$ is $\mathcal{B}([0,1])-\Sigma$-measurable and $c^*$ is $\check{\Sigma}-\mathcal{B}([0,1])$-measurable, where $\check{\Sigma}$ is the trace $\sigma$-algebra of $\Sigma$ on $c([0,1])$. 
We will now define a measure $\mu$ on $\Sigma$ which will have support 
in $\mathcal{S}$ only. In fact, it will have support only on those sequences consisting of projections of good crossed rectangles. To that end let $\mu_0(F_{1,0}):=1$ and, for  every   product set $A_k:=F_{n_0 ,0 }\times\dotsb\times F_{n_k,k}$  define recursively
\begin{align*}
\mu_{0,\dotsc,k}(A_k):=
\begin{cases}
0, &\text{ if } \neg(R_{n_1,1}\supseteq\dotsb\supseteq R_{n_k,k}),\\
0, &\text{ if }\exists l, 1\leq l\leq k \text{ s.t.~$R_{n_l,l}$ has }\\[-0.3em]
  &\qquad\qquad\quad \text{no good crossing},\\
\frac{\mu_{0,\dotsc,k-1}(F_{n_0 ,0 }\times\dotsb\times F_{n_{k-1},k-1})}{\xi(n,k-1)}, &\text{ otherwise,} 
\end{cases}
\end{align*}
where 
\begin{align*}
\xi(n,k):=\#\Big\{m\in \{(n-1)d_{k+1}+1,\ldots, & n d_{k+1}\} \colon \\[-0.6em]
&R_{m,k+1} \text{ has a good crossing in } R_{n,k} \Big\}
\end{align*} 
for $n=1,\ldots,N_k$.

As the family of measures $(\mu_{0 ,\dotsc,k})_{k\geq 1}$ clearly satisfies  the compatibility relations of the Kolmogorov extension theorem \cite[Theorem 14.35]{klenke}, there exists  a measure $\mu$ on $\Sigma$ such that $\mu$ restricted to the trace-$\sigma$-algebra of $\Sigma$ on $\prod_{0 \leq \tilde{k}\leq k}\{F_{n,\tilde{k}}:1\leq n\leq N_{\tilde{k}}\}$ yields $\mu_{0 ,\dotsc,k}$.
(Actually, already the Ionescu-Tulcea theorem \cite[Theorem 14.32]{klenke}, which is a ``smaller gun'' than Kolmogorov's theorem, suffices for this purpose.)

By means of the mapping $c^*$, we may transport the measure $\mu$ to $[0,1]$. The resulting pushforward measure $\nu:=c^*(\mu)$ then has support on $$F:=\left\{t\in [0,1]: t\in \bigcap_{k\geq 0 }F_{m(k,t),k}, \text{ all }R_{m(k,t),k}\text{ have good crossings}\right\}.$$
By the definition of $\mu$ and $\nu$, and since for all rectangle projections $R_{n,{k-1}}$ with good crossings there are at least $100^k-10$ descendant rectangles with good crossings, it follows that $\nu(\{t\})=0$
 for all $t\in[0,1]$.

Let now $s\in(0,1)$. Then choose the smallest $k_s\geq 0$ such that 
\begin{align}\label{eq:sexponent}
\frac{\log(100^{k_s}-10)}{\log(18\cdot 100^{k_s}+18)}>s.
\end{align}
We show that  
\begin{equation}\label{*kisl}
\text{$\nu(F_{n,l})\leq 21^l|F_{n,l}|$ for  $0\leq l<k_s$.}
\end{equation}
 The case
$l=0$ is obvious since  $\nu([0,1])=1\leq 21=21^1|F_{n,0}|$  trivially. 
 For  $1\leq l< k_s$ we have that
$$\nu(F_{n,l})\leq \frac{\nu(F_{n,l-1})}{100^l-10},$$
as there are at least $100^{l}-10$   rectangles in $F_{n,l-1}$ with good crossings. Using the induction hypothesis for $l-1$, we continue with 
$$\frac{\nu(F_{n,l-1})}{100^l-10}\leq \frac{21^{l-1}|F_{n,l-1}|}{100^l-10}=\frac{21^{l-1}w_{l-1}}{100^l-10}=\frac{21^{l-1}\prod_{1\leq \tilde{l}\leq l-1}\frac{1}{18+18\cdot 100^{\tilde{l}}}}{100^l-10}.$$
Since $l\ge 1$, the last   term is smaller than 
$$21^l\prod_{1\leq \tilde{l}\leq l}\frac{1}{18+18\cdot 100^{\tilde{l}}}=21^l|F_{n,l}|.$$
 From  $s\in (0,1)$, $1\leq l< k_s$ and \eqref{*kisl} it also follows that $\nu(F_{n,l})\leq 21^{{k_s}-1}|F_{n,l}|^s$

Next we consider $l\geq k_s$: Here we observe for $l=k_s:$ 
\begin{align*}
\nu(F_{n,k_s})\leq \frac{\nu(F_{n,k_s-1})}{100^{k_s}-10}\leq \frac{21^{k_s-1}|F_{n,k_s-1}|}{100^{k_s}-10}=  \frac{21^{k_s-1}\prod_{1\leq \tilde{l}\leq k_s-1}\frac{1}{18+18\cdot 100^{\tilde{l}}}}{100^{k_s}-10}, 
\end{align*}which can be bounded by
\begin{align*}
21^{k_s-1}\prod_{1\leq \tilde{l}\leq k_s}\left(\frac{1}{18+18\cdot 100^{\tilde{l}}}\right)^s=21^{k_s-1}|F_{n,k_s}|^s,
\end{align*}
which follows from $s\in (0,1)$ and \eqref{eq:sexponent}.
 This forms an induction start, such that for $l>k_s$, we get similarly,
\begin{align*}
\nu(F_{n,l})&\leq \frac{\nu(F_{n,l-1})}{100^l-10}\leq \frac{21^{k_s-1}|F_{n,l-1}|^s}{100^l-10}= \frac{21^{k_s-1}\prod_{1\leq \tilde{l}\leq l-1}\left(\frac{1}{18+18\cdot 100^{\tilde{l}}}\right)^s}{100^l-10}\\
& \leq 21^{k_s-1}\prod_{1\leq \tilde{l}\leq l}\left(\frac{1}{18+18\cdot 100^{\tilde{l}}}\right)^s=21^{k_s-1}|F_{n,l}|^s.
\end{align*}

Thus, for a given $s\in (0,1)$, we found  $k_s$ such that  for all $l$  
\begin{align}\label{eq:massdisprinc}
\nu(F_{n,l})\leq 21^{k_s-1}|F_{n,l}|^s.
\end{align}

Now we  show that there exists a constant $C_s$ such that 
$\nu(U)\leq C_s|U|^s$ for all measurable subsets $U\subseteq [0,1]$. 
Clearly, it is enough to show this for the case where $U$ is an interval, which we assume from now on.
There exists $l\in\N$ with $w_l<|U|\le w_{l-1}$. 
Now $U$ can be covered by at most 2 projections of rectangles of width $w_{l-1}$, 
say $U\subseteq F_{m-1,l-1}\cup F_{m,l-1}$. Moreover, we can approximate 
$U$ from outside by $F_{m_1,l}\cup\ldots \cup F_{m_1+a_1,l} \cup 
F_{m_1+a_1+1,l}\cup\ldots \cup F_{m_1+a_1+a_2+1,l}$, with 
$ F_{m_1+j_1,l}\subseteq  F_{m-1,l-1}$
and $ F_{m_1+a_1+j_2,l}\subseteq  F_{m,l-1}$ for all $0\le j_1 \le a_1$, $1\le j_2 \le a_2+1$, for an appropriate index $1\le m_1\le N_l$, and from inside by 
$F_{m_1+1,l}\cup\ldots \cup F_{m_1+a_1,l} \cup 
F_{m_1+a_1+1,l}\cup\ldots \cup F_{m_1+a_1+a_2,l}$.
Note that from these approximations we get 
\begin{equation}\label{eq:approx-outsidee}
a_1+a_2\le 18\cdot 100^l+18
\end{equation} 
and  
\begin{equation}\label{eq:approx-inside}
\max(1,a_1+a_2)w_l\le |U| \le (a_1+a_2+2) w_l\,.
\end{equation}

Assume, that the rectangles 
with projections
 $F_{m-1,l-1}, F_{m,l-1}$ admit good crossings (all other cases are easier to handle). Let then $B_1$ be the number of rectangles with good crossings on level $l$, with projections in $F_{m-1,l-1}$. Define the same number $B_2$ with respect to $F_{m,l-1}$ and note that $\min\{B_1,B_2\}\geq 100^l-10$. Denote the number of rectangle projections with good crossings of $F_{m_1,l},\dotsc, F_{m_1+a_1,l}$ by $b_1$ and those of $F_{m_1+a_1+1,l}\cup\ldots \cup F_{m_1+a_1+a_2+1,l}$ by $b_2$. Observe that 
 \begin{equation}\label{eq:approx-outside}
 b_1+b_2\le a_1+a_2+2\,.
 \end{equation} We use the
additivity of $\nu$ to get
\begin{align*}
\nu(U)\leq \sum_{i=0}^{a_1+a_2+1}\nu(F_{m_{1}+i,l})=\sum_{i=0}^{a_1}\nu(F_{m_{1}+i,l})+\sum_{i=a_1+1}^{a_1+a_2+1}\nu(F_{m_{1}+i,l}).
\end{align*}
By definition of $\nu$, the right-hand side of the above equality equals  
\begin{align*}
\frac{b_1}{B_1}\nu(F_{m-1,l})+\frac{b_2}{B_2}\nu(F_{m,l}).
\end{align*}
Now we may use \eqref{eq:massdisprinc} to continue with
\begin{align*}
&\frac{b_1}{B_1}\nu(F_{m-1,l})+\frac{b_2}{B_2}\nu(F_{m,l})\leq \frac{b_1}{B_1}21^{k_s-1}|F_{m-1,l}|^s+\frac{b_2}{B_2}21^{k_s-1}|F_{m,l}|^s\\
&= 21^{k_s-1}\left(\frac{b_1}{B_1}+\frac{b_2}{B_2}\right)w_{l-1}^s=21^{k_s-1}\left(\frac{b_1}{B_1}+\frac{b_2}{B_2}\right)\prod_{\tilde{l}=1}^{l-1}\frac{1}{(18\cdot 100^{\tilde{l}}+18)^s}\\
&=21^{k_s-1}\left(\frac{b_1}{B_1}+\frac{b_2}{B_2}\right)(18\cdot 100^l+18)^s\prod_{\tilde{l}=1}^{l}\frac{1}{(18\cdot 100^{\tilde{l}}+18)^s}\\
&=21^{k_s-1}\left(\frac{b_1}{B_1}+\frac{b_2}{B_2}\right)(18\cdot 100^l+18)^s w_l^s =:\text{\Large\bf \textasteriskcentered}.
\end{align*}
Using Equations \eqref{eq:approx-inside} and \eqref{eq:approx-outside}  and that $\min\{B_1,B_2\}\geq 100^l-10$, we can further estimate 
\begin{align*}
\text{\Large\bf \textasteriskcentered} &\leq 21^{k_s-1}\frac{b_1+b_2}{100^l-10}(18\cdot 100^l+18)^s \frac{|U|^s}{\max(1,a_1+a_2)^s}\\
&\leq 21^{k_s-1}\frac{a_1+a_2+2}{\max(1,a_1+a_2)^s}\frac{(18\cdot 100^l+18)^s}{100^l-10} |U|^s\,.
\end{align*}
Since
\begin{align*}
\frac{a_1+a_2+2}{\max(1,a_1+a_2)^s}&\le 
(a_{1}+a_{2})^{1-s}+ \frac{2}{\max(1,a_1+a_2)^s}
\le   (18\cdot 100^l+18)^{1-s}+2\\
&\le 3\cdot{(18\cdot 100^l+18)^{1-s}}\,,
\end{align*}
where we used  \eqref{eq:approx-outsidee} in the second inequality, we have altogether
\begin{align*}
\nu(U)
&\leq 3\cdot 21^{k_s-1}\frac{18\cdot 100^l+18}{100^l-10} |U|^s
\leq 3\cdot 21^{k_s} |U|^s\,.
\end{align*}

From the mass distribution principle it follows that $s\leq \hdim(F)$ for all $s\in (0,1)$, showing that $\hdim(F)=1$. 
As seen in Subsection \ref{sec:hoelder-example}, there is an at most countable set $E_f$ such that $$F\setminus E_f \subseteq\{t\in [0,1]:f(t)=g(t)\}.$$ Since a countable set does not change the Hausdorff-dimension, also $\hdim(F\setminus E)=1$. Since 
$F\setminus E \subseteq \{t\in [0,1]:f(t)=g_C(t)\}$, also
$$\hdim(\{t\in [0,1]:f(t)=g_C(t)\})=1.$$ 

It is not difficult to see that the function $g_C$ is not Hölder continuous for any exponent in $(0,1)$.
\end{proof}


\medskip

The mass distribution principle applied to the Hölder continuous function 
$g_H$, constructed in Subsection \ref{sec:hoelder-example}  does not deliver Hausdorff-dimension 1 of the intersection set. In this case, the rectangle width in each step is $w_k=\left(\frac{1}{234}\right)^k$, and in each step, for a rectangle with good crossing, one obtains at least $12-10$ good crossings.
 Hence, the mass distribution principle as above only works up to $s= \frac{\log(2)}{\log(234)}\approx 0.127$.

Recall that we have in total 14 ascending or descending chains within the first rectangle
$R_{1,0}$. A function $f$ with $\graph(f)$ having nonempty intersection with at least 3  rectangles within one ascending or descending chain needs
positive  
 variation for this. If it should touch more than 3 rectangles, it needs to have a variation of at least 
$\frac{m-3}{2}$ there. In any case, if there are  $m_1$  smaller rectangles in 
$R_{1,0}$ which have nonempty intersection with the graph of $f$, then the function $f$ needs a variation of at least 
\[
\frac{m_1 -3\cdot 14}{2}=\frac{ m_1 -42}{2}\,.
\] 
Conversely, this means $m_1 \le  42+2 v_1$, where $v_1=V_0^1(f)$.
Denote the variation of $f$ over the traversed rectangle $R_{n,1}$ by $v_{n,2}$, so that 
$\sum_{n=1}^{m_1}  v_{n,2} \le v_1$. Denote by $m_{n,2}$ the number  of smaller rectangles in $R_{n,1}$ which have nonempty intersection with the graph of $f$, and denote the sum of them for all $n$ by $m_{2}$, that is, 
\(m_2:=\sum_{n=1}^{m_1}m_{n,2}\). 
Then with the same argument as above,
$v_{n,2}\ge \frac{m_{n,2}-42}{20}$, so that $m_{n,2} \le 42+20 v_{n,2}$.
Summing over all $m_{n,2}$ we get 
\[
m_2\le \sum_{n=1}^{m_1} (42+20 v_{n,2})\le
42 m_1+ 20 v_1 < 42^2 +(2\cdot 42 +20)v_1\,.
\]
On the third level we get 
\[
m_3\le \sum_{n=1}^{m_2} 42+200 v_{n,3}\le 42 m_2+ 200 v_1\le 42^3 
+(2\cdot 42^2+20\cdot 42 + 200)v_1\,,
\]
and in general
\begin{align*}
m_k
&\le 42^k+2 v_1\sum_{j=1}^k 42^{k-j}10^{j-1}
=42^k\Big(1+\frac{1}{21} v_1\sum_{j=1}^k 42^{-(j-1)}10^{j-1}\Big)\\
&=42^k\bigg(1+\frac{1}{ 21} v_1\sum_{j=0}^{k-1} \Big(\frac{10}{42}\Big)^j\bigg)
\le 42^k\Big(1+\frac{1}{16} v_1 \Big)\,.
\end{align*}

As $v_1\le 1$, we see that the intersection set can be covered by at most $42^k \frac{17}{16}$  many intervals of width $w_k=\left(\frac{1}{234}\right)^k$.
 The $s$-dimensional Hausdorff-content 
$\mathscr{H}^s(\{t:f(t)=g_H(t)\})$  of the intersection set is thus bounded by
$$\lim_{k\to\infty}\frac{42^k \frac{17}{16}}{234^{ks}}=0$$   
for  $s>\frac{\log(42)}{\log(234)}$. Using the mass distribution principle as above, we get 
$$0.127\approx \frac{\log(2)}{\log(234)}\leq \hdim(\{t:f(t)=g_H(t)\})\leq \frac{\log(42)}{\log(234)}\approx 0.685\,.$$ 

In the special case of the constant function $f\equiv 0$, we see that it intersects exactly $2\cdot 13$  rectangles  out of the 234  rectangles in each step. Using the argument above, we get that 
$\hdim(\{t:g_H(t)=0\}) =\frac{\log(26)}{\log(234)}\approx 0.597$ .\\

We have seen   that for the $\frac{\log(10)}{\log(234)}$-H\"older continuous function $g_H$, we obtained an estimate for the Hausdorff-dimension of the intersection set 
$\{t\in [0,1]: f(t)=g_H(t)\}$  for every function $f$ with $f(0)=0$ and 
variation $V_0^1(f)\leq 1$. In the same spirit, we may ask, whether for a given Hölder exponent $\alpha\in (0,1)$ we can adapt our construction to produce a function $g_\alpha$ such that there are two constants $0\leq d(\alpha)\leq D(\alpha)\leq 1$ to control the Hausdorff-dimension by
$$d(\alpha)\leq \hdim(\{t\in [0,1]: f(t)=g_\alpha(t)\})\leq D(\alpha)$$  
for all functions $f$ with variation  less or equal  than 1. An immediate upper bound $1-\alpha$ for $D(\alpha)$ is given in \cite{rajalanote}, where level sets of H\"older functions are investigated, corresponding to constant functions.

\section{Application in metric topology}\label{sec:metric-top}

In this section we answer a question posed in \cite{leoste21}: Given a topological submanifold 
$\Theta\subseteq \R^d$, closed as a subset,  and a continuous function $f\colon \R^d\to\R$ which is Lipschitz continuous with respect to 
the intrinsic metric on $\R^d\setminus \Theta$, can we conclude that $f$ is Lipschitz continuous with respect to the Euclidean 
metric? The question is answered in the affirmative for the case of submanifolds with Lipschitz continuous charts,
cf.~\cite[Theorems 15 and 31]{leoste21}.  In this section we show that in 
general the answer is no, even for submanifolds with Hölder charts.

We start by recalling concepts from metric topology, including the notion of permeability of subsets of metric spaces,
 which was also introduced in \cite{leoste21}.  In Subsection \ref{sec:graphs-paths} we will show  that the earlier notion of permeability of the graph of 
 a continuous one-dimensional function is in harmony with the notion for subsets.
The example of a Hölder submanifold $\Theta\in \R^2$ with the property that intrinsic Lipschitz continuity of a function on $\R^2\setminus \Theta$ plus continuity on $\R^2$ does not imply Lipschitz continuity, is constructed in Subsection
\ref{sec:cantor-madness}.

\subsection{Intrinsically Lipschitz functions and permeable subsets of metric spaces}\label{sec:pwlm} 

Throughout this section, let $(M,d)$ be a metric space. To begin,  we recall some definitions for metric spaces. 
\begin{definition}[Path, arc, length]
 \phantom{ } 
\begin{enumerate}
\item A {\em path} in $M$ is a continuous mapping $\gamma\colon[a,b]\to M$.
We also say that $\gamma$ is a path in $M$ from $\gamma(a)$ to $\gamma(b)$. 
\item An injective path is called an {\em arc}.
\item If $\gamma\colon[a,b]\to M$ is a path in $M$, then its {\em length} $\ell(\gamma)$ is defined
as
\[
\ell(\gamma)
:=\sup\Big\{\sum_{k=1}^n d\big(\gamma(t_k),\gamma(t_{k-1})\big)\colon n\in\N,\,a=t_0<\ldots<t_n=b\Big\}\,.
\]
\end{enumerate}
\end{definition}


\begin{definition}[Intrinsic metric, length space]\label{def:intrinsic}
Let $E\subseteq M$ and $\Gamma(x,y)$ be the set of all paths of finite length in $E$ from $x$ to $y$. 
The {\em intrinsic metric} $\rho_E$ on $E$ is defined by
\[
\rho_E(x,y):=\inf\big\{\ell(\gamma)\colon \gamma\in \Gamma(x,y)\big\}\,, \qquad(x,y\in E)\,,
\]
with the convention $\inf\emptyset=\infty$.
The metric space $(M,d)$ is a {\em length space}, iff $\rho_M=d$.
 \end{definition}

Note that $\rho_E$ is not a proper metric in that it may assume  the value infinity. 
Of course, one could relate $\rho_E$ to a proper metric via
\[
\tilde \rho_E (x,y)
:=\begin{cases}
\frac{\rho_E(x,y)}{1+\rho_E(x,y)}\,, &\text{if } \rho_E(x,y)<\infty\,,\\
1\,, &\text{if }  \rho_E(x,y)=\infty\,.
\end{cases}
\] 
However, we stick to $\rho_E$, as it is the more natural choice and the extended co-domain does not lead to any difficulties.

It is readily checked that, if we allow $\infty$ as the value of a metric, then  
$(E,\rho_E)$ is  a length space.
%
%
%
%

\begin{definition}[Intrinsically Lipschitz continuous function]\label{def:pw-lip}
Let $E\subseteq M$, let $(\msr,d_\msr)$ be a metric space
and $f\colon M\to \msr$ be a function.
\begin{enumerate}
\item We call  $f$  {\em intrinsically $L$-Lipschitz continuous} 
on  $E$ 
iff
$f|_E\colon E\to \msr$ is Lipschitz continuous with respect to the
intrinsic metric $\rho_E$ on $E$ and $d_\msr$ on $\msr$ and 
$L$ is a Lipschitz constant for $f|_E$. 
\item We call  $f$  {\em intrinsically Lipschitz continuous} 
on  $E$ 
iff
$f$ is intrinsically $L$-Lipschitz continuous on $E$ for some $L$.
\item In the above cases we call $M\setminus E$ an {\em exception set (for intrinsic Lipschitz continuity)} of $f$. 
\end{enumerate}
\end{definition}

Note that if $\Theta\subseteq M$ is an exception set for intrinsical Lipschitz continuity that does not mean 
that Lipschitz continuity holds only on $E=M\setminus \Theta$. For example, if $f\colon M\to Y$ is intrinsically Lipschitz continuous, then every subset $\Theta\subseteq M$ is an exception set for intrinsical Lipschitz continuity of $f$.

A classical method for proving Lipschitz continuity of a differentiable function also works for intrinsic Lipschitz continuity: 
\begin{example}
Let $A\subseteq \R^d$ open and let $f\colon A\to \R$ be differentiable
with $\sup_{x\in A}\|\nabla f(x) \|<\infty$. Then $f$ is intrinsically Lipschitz continuous on $A$ with Lipschitz constant $\sup_{x\in A}\|\nabla f(x) \|$. 
 
A proof can be found in  \cite[Lemma 3.6 ]{sz2016b}.
\end{example}

\begin{example}\label{ex:intLip}
Consider the function
$f\colon\R^2\longrightarrow\R$, $f(x)=\|x\|\arg(x)$. Then $f$ is not Lipschitz continuous with respect to the Euclidean metric, since 
$$\text{$\lim_{h\to 0+}f(\cos(\pi-h),\sin(\pi-h))=-\pi$ and 
$\lim_{h\to 0+}f(\cos(\pi+h),\sin(\pi+h))= \pi$}.$$ 
It is readily checked, however, that $f$ is Lipschitz continuous on $E=\R^2\backslash\{x\in\R^2: x_1<0,x_2=0\}$ with respect to the intrinsic metric $\rho_E$.
Thus $f$ is intrinsically Lipschitz continuous with exception set $\Theta:=\{x\in\R^2: x_1<0,x_2=0\}$ in the sense of Definition \ref{def:pw-lip}.

\end{example}

\begin{remark}\label{rem:schlitz}
It is almost obvious, 
that every function $f\colon \R^2\to \R$, which is continuous on $\R^2$ and intrinsically Lipschitz continuous on 
$\R^2\setminus \{(x_1,x_2):x_1<0, x_2=0\}$, is Lipschitz continuous on $\R^2$.  
One can use that $\Theta:=\{(x_1,x_2):x_1<0, x_2=0\}$ does not pose a `hard' barrier, since every straight line 
connecting two points in $\R^2\setminus \Theta$ has at most one intersection point with $\Theta$ and so
one can conclude the Lipschitz continuity by approaching $\Theta$ from either side (we invite the reader to 
make this argument rigorous).   
\end{remark}

To make the elementary property of `not being a hard barrier' precise, we define at this point 
the notion of permeability.

\begin{definition}\label{def:permeable}
Let $E,\Theta\subseteq M$. 
\begin{enumerate}
\item The $\Theta$-{\em intrinsic metric} $\rho^\Theta_E$ on $E$ is defined by
\[
\rho^\Theta_E(x,y):=\inf\big\{\ell(\gamma)\colon \gamma\in \Gamma^\Theta(x,y)\big\} \,
\]
where $\Gamma^\Theta(x,y)$ is the set of all paths $\gamma\colon[a,b]\to M$ of finite length in $E$ from $x$ to $y$, such that $\overline{\{\gamma(t):t\in{[a,b]}\}\cap \Theta}$ is at most countable. 
(Again, we use the convention that $\inf\emptyset=\infty$.)
\item The $\Theta$-{\em finite intrinsic metric} $\rho^{\Theta,\textsc{fin}}_E$ on $E$ is defined by
\[
\rho^{\Theta,\textsc{fin}}_E(x,y):=\inf\big\{\ell(\gamma)\colon \gamma\in \Gamma^{\Theta,\textsc{fin}}(x,y)\big\} \,
\]
where $\Gamma^{\Theta,\textsc{fin}}(x,y)$ is the set of all paths $\gamma\colon[a,b]\to M$ of finite length in $E$ from $x$ to $y$, such that $\{\gamma(t):t\in{[a,b]}\}\cap \Theta$ is finite. 
\item We call $\Theta$ 
{\em permeable relative to $M$} iff $\rho_{M}=\rho_M^\Theta$.
\item We call $\Theta$ 
{\em finitely permeable relative to $M$} iff $\rho_{M}=\rho_M^{\Theta,\textsc{fin}}$.
\end{enumerate}
When the ambient space $(M,d)$ is understood and there is no danger of confusion, we simply
say $\Theta$ is (finitely) permeable.
\end{definition}

Note, that there seems to be a difference between the definitions of 
a permeable set and of a permeable graph. Theorem  \ref{th:harmony} will state
that in fact a continuous function has a permeable graph if and only if its graph is a permeable subset
of $\R^2$.

\begin{remark}\label{rem:permeable-simple}
A set $\Theta\subseteq M$ is (finitely) permeable iff for any $x,y\in M$ and every $\varepsilon>0$ there exists a path
$\gamma$ from $x$ to $y$ in $M$ with $\ell(\gamma)<\rho_M(x,y)+\varepsilon$ and such that
\(
\overline{\{\gamma(t):t\in{[a,b]}\}\cap \Theta}
\) is at most countable (finite). Clearly, every finitely permeable set is permeable.
\end{remark}

We now state a result from \cite{leoste21}.

\begin{theorem*}[{\cite[Theorem 15]{leoste21}}]
Let $\Theta\subseteq M$ be permeable, $(\msr,d_\msr)$ a  metric space. Then every continuous function $f\colon M\to \msr$,
which is intrinsically $L$-Lipschitz continuous on $E=M\setminus \Theta$, is intrinsically $L$-Lipschitz continuous on the whole of $M$.
\end{theorem*}

In particular, if $M$ is a length space, we get the following corollary.
\begin{corollary*}[{\cite[Corollary 20]{leoste21}}]
Let $M$ be a length space and let $\Theta\subseteq M$ be permeable. Then every continuous function $f\colon M\to \msr$ 
into a  metric space $(\msr,d_\msr)$,
which is intrinsically $L$-Lipschitz continuous on $E=M\setminus \Theta$, is $L$-Lipschitz continuous on the whole of $M$ (i.e., with respect to $d$).
\end{corollary*}

Revisiting  Remark \ref{rem:schlitz}, we see that the situation is more interesting if we replace the set 
$\R^2\setminus \{(x,0):x<0\}$ by $\R^2\setminus \Theta$ with $\Theta=\{(x,g(x)):x\le 0\}$, where $g$ is some continuous function on $(-\infty,0]$.
By the above theorem, every continuous function $f\colon \R^2\to \R$, which is intrinsically Lipschitz on 
$\R^2\setminus \Theta$ is automatically Lipschitz on $\R^2$, if $\Theta$ is permeable, which in turn is 
the case iff $g$ has permeable graph, by Theorem  \ref{th:harmony}.

From Theorem \ref{thm:typical} we know, that for a typical $g$, $\Theta$ is permeable, but from 
Corollary \ref{thm:hoelder-ex} we also know that not every Hölder function $g$ pertains this property.

%
%
%
%

The study of permeability of subsets of $\R$  is very tidy: A subset $\Theta$ is permeable iff it has countable 
closure. Further this is equivalent to the property that every continuous 
intrinsically Lipschitz continuous function with exception set $\Theta$ is Lipschitz continuous. 

\begin{theorem*}[{\cite[Theorem 23]{leoste21}}]
Let $\Theta\subseteq \R$. Then $\Theta$ has countable closure if and only if
for all intervals $I\subseteq \R$ and all functions $f\colon I\to \R$ 
the properties
\begin{itemize}
\item $f$ is intrinsically Lipschitz continuous with exception set $\Theta$,
\item $f$ is continuous,
\end{itemize}
imply that $f$ is Lipschitz continuous on $I$.
\end{theorem*}

Consider $\Theta=(\R\setminus \Q)^2$. It has been shown in \cite[Proposition 26]{leoste21} that $\Theta$ is impermeable, but
every continuous function $f\colon \R^2\to \R$ which is intrinsically Lipschitz continuous on $\R^2\setminus \Theta$ is 
Lipschitz continuous. However, one cannot conclude $L$-Lipschitz continuous from $L$-intrinsically Lipschitz continuous,
rather, in this particular example, the Lipschitz constant has to be multiplied with $\sqrt{2}$.

\subsection{Graphs of functions and images of paths}\label{sec:graphs-paths}

This subsection contains technical results about the connection between  the notions of
continuous functions with permeable graphs and continuous functions whose 
graphs are permeable sets. The main result is Theorem  \ref{th:harmony}, which states that these concepts coincide.

\begin{definition}
Let $f\colon [a,b]\to \R$. 
\begin{itemize} 
\item We say $f$ is {\em regulated}, if its right-sided limit 
$f(x+):=\lim_{y\searrow x} f(y)$ exists in every 
$x\in[a,b)$, and its left-sided limit $f(x-):=\lim_{y\nearrow x} f(y)$ exists in every 
$x\in(a,b]$.
\item If $f$ is regulated, we define $f(a-):=f(a)$ and $f(b+):=f(b)$.
\item If $f$ is regulated we define its {\em connected graph}
\[
\cgraph(f):=\bigcup_{x\in [a,b]}\conv\big(\{f(x),f(x+),f(x-)\}\big)\,.
\] 
\item We call a function {\em unfrayed}, if it is regulated and $f(x)\in \conv(\{f(x+),f(x-)\})$ for all $x\in [a,b]$.
\end{itemize}
\end{definition}

\begin{proposition}\label{prop:regulated-var}
Let $f$ be a regulated function. Then 
\begin{align*}
V(f)&=\sup\sum_{k=1}^n \left(\big|f(x_{k-1}+)-f(x_{k-1})\big|+\big|f(x_k-)-f(x_{k-1}+)\big|+\big|f(x_k)-f(x_k-)\big| \right)\\
\text{and}&\\
\ell(f)&=\sup\sum_{k=1}^n \bigg(\big|f(x_{k-1}+)-f(x_{k-1})\big|
+\sqrt{(x_k-x_{k-1})^2+\big(f(x_k-)-f(x_{k-1}+)\big)^2}\\
&\qquad\qquad\qquad\qquad\qquad\qquad\qquad\qquad\qquad\qquad\qquad+\big|f(x_k)-f(x_k-)\big| \bigg)\,,
\end{align*}
where the suprema are taken over all partitions $a= x_0<\ldots<x_n=b$ of $[a,b]$.
\end{proposition}

\begin{proof}  The proof is quite straightforward and is left to the reader.
\exclude{
Let $a= x_0<\ldots<x_n=b$ be a partition of $[a,b]$, and let $\varepsilon>0$. Then there exist
$x_0^+,\ldots,x_{n-1}^+$ and $x_1^-,\ldots,x_{n}^-$ with $x_{k-1}<x_{k-1}^+<x_k^-<x_k$ for all $k$ and
$\big|f(x_k+)-f(x_k^+)\big|<\frac{\varepsilon}{4(n+1)}$ and $\big|f(x_k-)-f(x_k^-)\big|<\frac{\varepsilon}{4(n+1)}$.
Then 
\begin{align*}
\lefteqn{\sum_{k=1}^n \left(\big|f(x_{k-1}+)-f(x_{k-1})\big|+\big|f(x_k-)-f(x_{k-1}+)\big|+\big|f(x_k)-f(x_k-)\big| \right)}\\
&\le\sum_{k=1}^n \left(\big|f(x_{k-1}^+)-f(x_{k-1})\big|+\big|f(x_k^-)-f(x_{k-1}^+)\big|+\big|f(x_k)-f(x_k^-)\big| \right)+\varepsilon\\
&\le V_a^b(f)+\varepsilon\,.
\end{align*}
For the other inequality let $\varepsilon>0$. There exists a partition 
$a=x_0<\ldots<x_n=b$ of $[a,b]$
with 
\begin{align*}
V_a^b(f)&< \sum_{k=1}^n  \big|f(x_k)-f(x_{k-1})\big|+\varepsilon\\
&\le \sum_{k=1}^n  \big(\big|f(x_{k-1}+)-f(x_{k-1})\big|+\big|f(x_k-)-f(x_{k-1}+)\big|\\
&\qquad\qquad\qquad\qquad\qquad\quad+\big|f(x_k)-f(x_{k}-)\big|\big)+\varepsilon\,.
\end{align*}
From this the claim follows. The proof for the second statement is similar.
}
\end{proof}

The next proposition states that there is a one-one relation between  right-continuous scalar functions 
of bounded variation and certain paths of finite length connecting the endpoints of the function's 
graph in the plane. Basically, the path traces the graph and the jumps of the function. 
Recall that the coordinate functions of a given function 
$\gamma\colon [0,1]\to \R^2$
are denoted by $\gamma_1,\gamma_2$.

\begin{proposition}\label{prop:variation-length}
Let $a,b\in\R^2$, $a<b$. Let an arc $\gamma\colon[0,1]\to\R^2$ be given whose
first component  $\gamma_1$ is 
nondecreasing with $\gamma_1(0)=a$ and $\gamma_1(1)=b$.
Then there is an unfrayed function 
$f\colon[a,b]\to\R$  such that 
$\cgraph(f)=\gamma([0,1])$. 

Conversely, if an unfrayed function $f\colon[a,b]\to\R$  is given, then there exists 
an arc $\gamma\colon[0,1]\to\R^2$ whose
first component  $\gamma_1$ is 
nondecreasing with $\gamma_1(0)=a$ and $\gamma_1(1)=b$, and for which 
 $\gamma([0,1])=\cgraph(f)$. 
 
In both cases $\ell_a^b(f)=\ell(\gamma)$ and $V_a^b(f)=V_0^1(\gamma_2)$.
\end{proposition}

\begin{proof}
 We  define a function $\gamma_1^+\colon[a,b]\to \mathbb{R}$
by
\begin{align*}
\gamma_1^+(x):=\begin{cases}
\inf\{t\ge 0\colon \gamma_1(t)>x\}, & \text{ if } x\in [a,b)\,,\\
1 , & \text{ if } x=b
\end{cases}
\end{align*}
Since $\gamma_1$ is continuous and non-decreasing, $\gamma_1^+\colon [a,b]\to \mathbb{R}$ is
a non-decreasing function. Moreover, $\gamma_1(\gamma_1^+(x))=x$ for  all  $x\in [a,b]$.  
Note that $\gamma_1^+$ is right-continuous on $(a,b]$ and $\gamma(x-)$ exists on $(a,b]$.
Define 
\[
f(x):=\begin{cases}
\gamma_2(\gamma_1^+(x)), & \text{ if } x\in (a,b]\,,\\
\gamma_2(0), & \text{ if } x=a\,.
\end{cases}
\] 
Noting that $f$ has jumps precisely on intervals where $\gamma_1$ is constant, it is not difficult to see that 
$f$ is unfrayed with 
$\cgraph(f)=\gamma([0,1])$. 

 The opposite direction of the proof is also quite simple. To obtain $\gamma$, one needs to connect
the vertical jumps of the unfrayed function $f$ by vertical line segments.
Further details are left to the reader.  
\end{proof}

{
\begin{proposition}\label{prop:monComponentPath}
Let $a<b$, $f_a,f_b\in \R$, and let $\gamma\colon[0,1]\to\R$ be a path connecting $(a,f_a)$ and $(b,f_b)$. Then there is an arc $\bar{\gamma}\colon[0,1]\to\R$ connecting the same points with $\bar{\gamma}_1$ monotonically increasing and such that $\ell(\bar{\gamma})\leq \ell(\gamma)$  and $V_0^1(\bar\gamma_2)\le V_0^1(\gamma_2)$. 
Furthermore, there are at most countably many disjoint intervals $([s_k,t_k])_{0\le k< K}$, 
$K\in \N_0\cup \{\infty\}$,  such that 
$\bar{\gamma}_1$ is constant on each of the $[s_k,t_k]$  and $$\gamma|_{[0,1]\setminus\bigcup_{k=0}^{K-1} (s_k,t_k)}=\bar{\gamma}|_{[0,1]\setminus\bigcup_{k=0}^{K-1} (s_k,t_k)}.$$
\end{proposition}
\begin{proof} 
 Let $\bar{\gamma}_1\colon[0,1]\to[a,b]$ be defined by $\bar{\gamma}_1(t):=\sup_{0\leq s\leq t}\gamma_1(t)$. Then $\bar{\gamma}_1$ is continuous and monotonically increasing and therefore there are  at most countably many (pairwise disjoint) intervals $([s_k,t_k])_{0\le k<K}$ on which it is constant, and on the complement of $\bigcup_{ k=0}^{K-1} [s_k,t_k]$ it is strictly increasing, for some $K\in \N_0\cup \{\infty\}$.  Then we set
  \begin{align*}
\bar{\gamma}_2(t):=
\begin{cases}
\gamma_2(t), & \text{if }t\notin \bigcup_{k=0}^{K-1 } (s_k,t_k),\\
\gamma_2(s_k)+\frac{\gamma_2(t_k)-\gamma_2(s_k)}{t_k-s_k}(t-s_k),&  \text{if } t\in(s_k,t_k).
\end{cases}
  \end{align*}
  Then $\bar{\gamma}:=(\bar{\gamma}_1,\bar{\gamma}_2)$ is an arc, since 
  $\bar\gamma_2$ is strictly monotone on the intervals where $\bar\gamma_1$ is constant, and $\bar{\gamma}(0)=(a,f_a)$ and $\bar{\gamma}(1)=(b,f_b)$. 
 Using this definition we leave to the reader the easy task to verify that
$\bar{\gamma}$ satisfies the claim of the proposition.
\end{proof}
}
We give another simple result on the relation between the length of a path and 
the total variation of its component functions.

\begin{proposition}\label{prop:estimate-len-var}
Let $\gamma\colon [a,b]\to \R^2$ be a path in $\R^2$.   Then for $j\in \{1,2\}$
\[
V_a^b(\gamma_j)\le \ell(\gamma)\le V_a^b(\gamma_1)+V_a^b(\gamma_2)\quad
\text{and}\quad
V_a^b(\gamma_1)+V_a^b(\gamma_{2})\leq \sqrt{2}\ell(\gamma).
\]
Let
$f\colon[a,b]\to\R$ be a function. Then 
\[V_a^b(f)\le \ell_a^b(f)\leq b-a+V_a^b(f)
\quad\text{and}\quad b-a+V_a^b(f)\le \sqrt{2}\ell_a^b(f)\,.
\]

\end{proposition}

 Again we leave the simple proof to the reader.
\exclude{
\begin{proof} 
Consider a partition $a=t_0<t_1<\ldots<t_n=b$ of $[a,b]$. We write $x_k:=\gamma_1(t_k)$ and $y_k:=\gamma_2(t_k)$, $k=0,\ldots,n$.
\begin{align*}
\sum_{k=1}^n |\gamma_j(t_k)-\gamma_j(t_{k-1})|
&=\sum_{k=1}^n \|\gamma(t_k)-\gamma(t_{k-1})\|\\
&=\sum_{k=1}^n \sqrt{(x_k-x_{k-1})^2+(y_k-y_{k-1})^2}\\
&\le\sum_{k=1}^n \big(|x_k-x_{k-1}|+|y_k-y_{k-1}|\big) \le V_a^b(\gamma_1)+V_a^b(\gamma_2)\,.
\end{align*}
Taking the supremum over all partitions finishes the proof.
\end{proof}}

\begin{proposition}\label{prop:variationLeftLimit}
Let $f\colon[a,b]\to\R$ have bounded variation. Then $\overrightarrow{f}\colon[a,b]\to\R, t\mapsto f(t-)$ exists, $V_a^b(\overrightarrow{f})\leq V_a^b(f)$  and $\ell_a^b(\overrightarrow{f})\leq \ell_a^b(f)$.
\end{proposition}

\begin{proof}
As $f$ is a function of bounded variation, the left and right limits exist, so $\overrightarrow{f}$ can be defined and is unfrayed.
By Proposition \ref{prop:regulated-var},
\begin{align*}
V_0^1(f)&=\sup\sum_{k=1}^n \left(\big|f(x_{k-1}+)-f(x_{k-1})\big|+ \big|f(x_k-)-f(x_{k-1}+)\big|+\big|f(x_k)-f(x_k-)\big|\right)\\
&=\sup\sum_{k=1}^{n-1} \left(\big|f(x_{k}+)-f(x_{k})\big|+\big|f(x_k)-f(x_k-)\big|\right)
+\sum_{k=1}^n \big|f(x_k-)-f(x_{k-1}+)\big|\\
&\quad\qquad+\big|f(x_{0}+)-f(x_{0})\big|+\big|f(x_n)-f(x_n-)\big|\\
&\ge\sup\sum_{k=1}^{n-1} \left(\big|f(x_{k}+)-f(x_k-)\big|\right)
+\sum_{k=1}^n \big|f(x_k-)-f(x_{k-1}+)\big|\\
&\quad\qquad+\big|f(x_{0}+)-f(x_{0})\big|+\big|f(x_n)-f(x_n-)\big|\\
&\ge\sup\sum_{k=1}^{n-1} \left(\big|\overrightarrow{f}(x_{k}+)-\overrightarrow{f}(x_k-)\big|\right)
+\sum_{k=1}^n \big|\overrightarrow{f}(x_k-)-\overrightarrow{f}(x_{k-1}+)\big|\\
&\quad\qquad+\big|\overrightarrow{f}(x_{0}+)-\overrightarrow{f}(x_{0})\big|+\big|\overrightarrow{f}(x_n)-\overrightarrow{f}(x_n-)\big|\\
&=V_0^1(\overrightarrow{f})\,,
\end{align*}
as $\big|f(x_n)-f(x_n-)\big|\ge 0=\big|\overrightarrow{f}(x_n)-\overrightarrow{f}(x_n-)\big|$.

A similar argument shows the claim about the lengths of the graphs.
\end{proof}

\begin{proposition}\label{prop:permeableOnSubinterval}
Let $g\colon[0,1]\to\R$ be a function with permeable graph. Then for any subinterval $[a,b]\subseteq[0,1]$, $g|_{[a,b]}$ has a permeable graph on $[a,b]$.
\end{proposition}
\begin{proof}
Let $y\in \R$ and $\delta>0$ be given. 
Since $g$ has a permeable graph, there is a function $\hat{f}\colon[0,1]\to\R$ such that $\hat{f}(0)=\hat{f}(1)=y$, $V_0^1(\hat{f})<\frac{\delta}{3}$ and $\overline{\{t\in [0,1]:g(t)=\hat{f}(t)\}}$ is countable. 
Hence, $|\hat{f}(a)-y|<\frac{\delta}{3}$ and $|\hat{f}(b)-y|<\frac{\delta}{3}$. 
Define $f\colon[a,b]\to\R$ by
\begin{align*}
t\mapsto\begin{cases}
y, & t\in \{a,b\},\\
\hat{f}(t), & t\in (a,b).
\end{cases}
\end{align*}
Then $V_a^b(f)\leq |\hat{f}(a)-y|+V_a^b(\hat{f})+|\hat{f}(b)-y|<\delta$ and 
\[\overline{\{t\in [a,b]:g(t)={f}(t)\}}\subseteq \overline{\{t\in [0,1]:g(t)=\hat{f}(t)\}}\cup\{a,b\}\,,\]
so $\overline{\{t\in [a,b]:g(t)={f}(t)\}}$ is countable.
\end{proof}

{
\begin{theorem}\label{th:harmony}
Let $g\colon [a,b]\to \R$ be a continuous function. Then $g$  has permeable graph iff the graph of $g$ is a permeable subset
of $\R^2$.
\end{theorem}
}

{
\begin{proof}
We assume that $a=0, b=1$. First we show that the permeable graph property of a continuous function $g\colon[0,1]\to\R$ implies the permeability of $\graph(g)$ as a subset of $\R^2$.

Any vertical line segment connecting two points in the plane has at most one intersection point with the graph of $g$ and is trivially their shortest connection. Hence we only concentrate on line segments between points $(x_1,y_1), (x_2,y_2)$ with $x_1\neq x_2$. In particular, without loss of generality, we may assume that 
$ (x_1,y_1)=(0,f_a), (x_2,y_2)=(1,f_b)$.

Let $\delta>0$ be given. The graph of $g$ is a zero set with respect to the two-dimensional Lebesgue-measure $\lambda^{(2)}$, which one can see by applying Fubini's theorem. Hence, also 
\begin{align*}
\int_{\left\{(t,u+f_a+(f_b-f_a)t):\, u\in \left[-\frac{\delta}{4},\frac{\delta}{4}\right], t\in [0,1]\right\}}\mathbbm{1}_{\graph(g)} d\lambda^{(2)}=0.
\end{align*}
Again by Fubini's theorem we get that for some $u\in \left[-\frac{\delta}{4},\frac{\delta}{4}\right]$ (and actually $\lambda$-almost all $u$),  $\lambda(\{t\in [0,1]: u+f_a+(f_b-f_a)t=g(t)\})=0$. Moreover, we may choose $u$ so that $u+f_a\neq g(0)$ and $u+f_b\neq g(1)$. We define $\phi(t):=u+f_a+(f_b-f_a)t$ and $F:=\{t\in [0,1]:\phi(t)=g(t)\}$. As $[0,1]\setminus F$ is an open subset of $[0,1]$, it can be written as disjoint union $[0,a_0)\cup(b_0,1]\cup\bigcup_{k=1}^\infty (a_k,b_k)$. Since $F$ has measure zero, we can find $K>0$ such that $$F\subseteq[0,1]\setminus\left([0,a_0)\cup(b_0,1]\cup\bigcup_{k=1}^K (a_k,b_k)\right)=:\bigcup_{k=0}^K [c_k,d_k]$$ and $\lambda\left(\bigcup_{k=0}^K [c_k,d_k]\right)=\sum_{k=0}^K(d_k-c_k)< \frac{\delta}{4(\|(1,f_b-f_a)\|_1-\|(1,f_b-f_a)\|_2)+1}$. As $g$ is permeable on each interval $[c_k,d_k]$ by Proposition \ref{prop:permeableOnSubinterval}, it follows that there is a function $f_{c_k,d_k}$, such that $V_{c_k}^{d_k}(f_{c_k,d_k})<\frac{\delta}{4(K+1)}$, $f_{c_k,d_k}(c_k)=f_{c_k,d_k}(d_k)=\phi(c_k)$ and $\overline{\{t\in [0,1]:f_{c_k,d_k}(t)=g(t)\}}$ is countable. We define the function $f\colon [0,1]\to\R$ by
\begin{align*}
t\mapsto\begin{cases}
\phi(t), & t\notin \bigcup_{k=0}^K (c_k,d_k],\\
f_{c_k,d_k}(t), & t\in (c_k,d_k].
\end{cases}
\end{align*}
Then $f(0)=f_a+u, f(1)=f_b+u$ and $\overline{\{t\in [0,1]:f(t)=g(t)\}}$ is countable as we defined the function $f$ piecewisely on finitely many disjoint intervals.

Let now $\overrightarrow{f}(t):=f(t-)$ for all $t\in [0,1]$ (we can take left limits as $f$ is of bounded variation). The operation only takes effect on the intervals $(c_k,d_k], k=0,\dotsc,K$, where 
$\overrightarrow{f}(t)=\overrightarrow{f}_{\!\!c_k,d_k}(t):=\overrightarrow{f}_{\!\!c_k,d_k}(t-)$, on the other intervals $\overrightarrow{f}=f$ holds. By Propositions \ref{prop:regulated-var} and \ref{prop:variationLeftLimit}, taking into account the possible jumps at the locations $d_k$, we get
\begin{align*}
&\ell_0^1(\overrightarrow{f})\leq\ell_0^1(f)
\leq\ell_0^{a_0}(\phi)+\ell_{b_0}^1(\phi)+\sum_{k=1}^K\ell_{a_k}^{b_k}(\phi)+\sum_{k=0}^K\ell_{c_k}^{d_k}(f_{c_k,d_k})+\sum_{k=0}^K|\phi(d_k)-\phi(c_k)|\\
&\leq \ell_0^{a_0}(\phi)+\ell_{b_0}^1(\phi)+\sum_{k=1}^K\ell_{a_k}^{b_k}(\phi)+\sum_{k=0}^K\left((d_k-c_k)+V_{c_k}^{d_k}(f_{c_k,d_k})+|\phi(d_k)-\phi(c_k)|\right),
\end{align*}
where we used Proposition \ref{prop:estimate-len-var}. We continue the above estimation with
\begin{align*}
\ell_0^1(\overrightarrow{f})&=\sqrt{a_0^2+(f_a+u-\phi(a_0))^2}+\sqrt{(1-b_0)^2+(f_b+u-\phi(b_0))^2}\\
&\quad+\sum_{k=1}^K\sqrt{(b_k-a_k)^2+(\phi(b_k)-\phi(a_k))^2}\\
&\quad+\sum_{k=0}^{K}\left((d_k-c_k)+|\phi(d_k)-\phi(c_k)|+V_{c_k}^{d_k}(f_{c_k,d_k})\right)\\
&< \sqrt{a_0^2+(f_a+u-\phi(a_0))^2}+\sqrt{(1-b_0)^2+(f_b+u-\phi(b_0))^2}\\
&\quad+\sum_{k=1}^K\sqrt{(b_k-a_k)^2+(\phi(b_k)-\phi(a_k))^2}\\
&\quad+\sum_{k=0}^{K}\left((d_k-c_k)+|\phi(d_k)-\phi(c_k)|+\frac{\delta}{4(K+1)}\right).
\end{align*}
By the form of $\phi$, the last sum can be written as
\begin{align*}
&\sum_{k=0}^{K}\biggl(\sqrt{(d_k-c_k)^2+(\phi(d_k)-\phi(c_k))^2}\\
&\quad\quad+(\|(1,f_b-f_a)\|_1-\|(1,f_b-f_a)\|_2)(d_k-c_k)+\frac{\delta}{4(K+1)}\biggr).
\end{align*}

Glueing the pieces of the line segments together, we infer
$$\ell_0^1(\overrightarrow{f})< \|(0,f_a)-(1,f_b)\|_2+\frac{\delta}{2}.$$

Since $\overrightarrow{f}$ is left-continuous and $g$ is continuous, if $(t_n)_{n\geq 0}$ with $f(t_n)=g(t_n)$ tends to ${ \bar t}\in [0,1]$ from below, then $f({ \bar t}-)=g({ \bar t})$ and therefore ${ \bar t}\in\overline{\{t\in [0,1]:f(t)=g(t)\}}$. Hence the set $\overline{\{t\in [0,1]:\overrightarrow{f}(t)=g(t)\}}$ is also countable.

By our choice of $u$, $a_0>0$ and $b_0<1$, thus $\overrightarrow{f}(0)=f_a+u$ and $f(1)=f(1-)=f_b+u$.

Due to Proposition \ref{prop:variation-length}, there is an arc $\hat{\gamma}\colon[0,1]\to \R^2$ connecting $(0,f_a+u)$ and $(1,f_b+u)$ and $\ell(\hat{\gamma})=\ell_0^1(\overrightarrow{f})< \|(0,f_a)-(1,f_b)\|_2+\frac{\delta}{2}$ and
$\hat{\gamma}([0,1])=\cgraph(\overrightarrow{f})$. Thus additionally to $\overline{\{t\in [0,1]:\overrightarrow{f}(t)=g(t)\}}$, we get the closure of 
$$\Big\{t\!\in\! [0,1]\colon\!\!\overrightarrow{f}(t)\!\neq\! \overrightarrow{f}(t+)\text{ and }\min\big(\overrightarrow{f}(t),\overrightarrow{f}(t+)\big)\leq g(t)\leq \max\big(\overrightarrow{f}(t),\overrightarrow{f}(t+)\big)\!\Big\},$$
which we call $J_g$. Let $J:=\big\{t\in [0,1]:\overrightarrow{f}(t)\neq \overrightarrow{f}(t+)\big\}$ be the set of jump locations of $\overrightarrow{f}$, which is countable, as $\overrightarrow{f}$ is of bounded variation. Take a sequence $(t_n)_{n\geq 0}$ in $J_g$ converging to some $\bar t$. Then there are two possibilities: Either, there is a jump at $\bar t$, then $\bar t$ is contained in $J$. Or, there is no jump at $\bar t$, then $\overrightarrow{f}$ is continuous at $\bar t$. 
In the latter case, by the continuity of $\overrightarrow{f}$ and $g$, it follows that $\overrightarrow{f}(\bar t)=g(\bar t)$ and hence $\bar t\in \overline{\{t\in [0,1]:\overrightarrow{f}(t)=g(t)\}}$. In any case, we have that
$$\overline{J_g}\subseteq J\cup\overline{\{t\in [0,1]:\overrightarrow{f}(t)=g(t)\}},$$
and the right hand side is countable. We infer that $\overline{\hat{\gamma}([0,1])\cap \graph(g)}$ is also countable.

Extending $\hat\gamma$ to an arc $\gamma$ by additionally connecting vertically $(0,f_a)$ with $(0,f_a+u)$ and $(1,f_b)$ with $(1,f_b+u)$, we get, as $u\in \left[\frac{\delta}{4},\frac{\delta}{4}\right]$,
\begin{align*}
&\|(0,f_a)-(1,f_b)\|_2\leq \ell(\gamma)\leq \ell(\hat{\gamma})+\frac{\delta}{2}=\ell_0^1(\overrightarrow{f})+\frac{\delta}{2}\\
&<\|(0,f_a)-(1,f_b)\|_2+\frac{\delta}{2}+\frac{\delta}{2}=\|(0,f_a)-(1,f_b)\|_2+\delta,
\end{align*}
as desired.
\bigskip

Conversely, assume that $\graph(g)$ is a permeable set. Let $y\in \R$ and $\delta>0$. As for the implication before, using a Fubini argument, we infer that there is $u\in [y-\frac{\delta}{4},y+\frac{\delta}{4}]$ such that $\lambda(\{t\in[0,1]:u=g(t)\})=0$ and $g(0)\neq u\neq g(1)$. Set again $F:=\{t\in[0,1]:u=g(t)\}$. Then we can find $K>0$ such that $$F\subseteq[0,1]\setminus\left([0,a_0)\cup(b_0,1]\cup\bigcup_{k=1}^K (a_k,b_k)\right)=:\bigcup_{k=0}^K [c_k,d_k]$$ and $\lambda\left(\bigcup_{k=0}^K [c_k,d_k]\right)=\sum_{k=0}^k(d_k-c_k)<\frac{\delta}{4}$. As $g$ is permeable, we can find paths $\gamma_{c_k,d_k}\colon[0,1]\to\R$, connecting $(c_k,u)$ and $(d_k,u)$ such that $\ell(\gamma_{c_k,d_k})<(d_k-c_k)+\frac{\delta}{4(K+1)}$ and $\overline{\gamma_{c_k,d_k}([0,1])\cap\graph(g)}$ is countable. By Proposition \ref{prop:monComponentPath}, we can find for each $\gamma_{c_k,d_k}$ an arc $\bar{\gamma}_{c_k,d_k}$, connecting the same points and having monotone first component. We then set $\gamma\colon[0,1]\to\R$,
\begin{align*}
\gamma(t):=\begin{cases}
(t,u)\,,& \text{if } t\notin \bigcup_{k=0}^K[c_k,d_k],\\
\bar{\gamma}_{c_k,d_k}\!\!\left(\frac{t-c_k}{d_k-c_k}\right), & \text{if }t\in [c_k,d_k].
\end{cases}
\end{align*}
 By Proposition \ref{prop:variation-length} there exists an unfrayed function $f$ with $\cgraph(f)=\bar\gamma([0,1])$ and $V_0^1(f)= V_0^1(\bar{\gamma}_2)$. Note that $f$ jumps at $t$ if and only if $t$ is in an interval where the paths $\bar{\gamma}_{c_k,d_k}$ have constant first component, so $\{t\in [0,1]:f(t)=g(t)\}$ is determined by the closed sets where $\bar{\gamma}_{c_k,d_k}$ equals $\gamma_{c_k,d_k}$, intersected with the graph of $g$ (which is also closed). Thus, $\overline{\{t\in [0,1]:f(t)=g(t)\}}$ is countable. The variation of $f$ can be estimated by 
 \begin{align*}
 V_0^1(f)=\sum_{k=0}^KV_{c_k}^{d_k}(f)\leq \sum_{k=0}^KV_{c_k}^{d_k}(\gamma_2)\,.
 \end{align*}
 Proposition \ref{prop:estimate-len-var} implies that 
 \begin{align*}
 &V_{c_k}^{d_k}(\gamma_2)=V_{0}^{1}\big((\gamma_{c_k,d_k})_2\big)\leq\sqrt{2}\ell({\gamma}_{c_k,d_k})
 -V_0^1\big((\gamma_{c_k,d_k})_1\big)
 = \sqrt{2}\ell(\bar{\gamma}_{c_k,d_k})-(d_k-c_k)\\
 &\leq \sqrt{2}\ell({\gamma}_{c_k,d_k})-(d_k-c_k),
 \end{align*}
 since $\gamma_{c_k,d_k}$ has increasing first component, and by Proposition \ref{prop:monComponentPath}, $\ell(\bar{\gamma}_{c_k,d_k})\leq \ell({\gamma}_{c_k,d_k})$. As $\ell(\gamma_{c_k,d_k})<(d_k-c_k)+\frac{\delta}{4(K+1)}$, we get that
 \begin{align*}
 V_0^1(f)<\sum_{k=0}^K(\sqrt{2}-1)(d_k-c_k)+\frac{\delta}{4}\leq \frac{\delta}{2}.
 \end{align*}
 Put 
 $$\bar{f}(t):=\begin{cases}
 y,&\text{if }t\in\{0,1\},\\
 f(t),&\text{else.}
 \end{cases}$$
 Since $u\in \left[y-\frac{\delta}{4},y+\frac{\delta}{4}\right]$, we get that $V_0^1(\bar{f})< \delta$ and the set $\overline{\{t\colon [0,1]: \bar{f}(t)=g(t)\}}$ is countable, as asserted.
\end{proof}
}

\subsection{A non-Lipschitz intrinsically Lipschitz function}\label{sec:cantor-madness}

The main goal of this section is Theorem \ref{*thmintr}. This theorem provides
 an example of a continuous function $f\colon \R^2\to \R$
which is intrinsically Lipschitz continuous on $\R^2\setminus \Theta$ for some impermeable subset 
$\Theta\subset \R^2$, but not Lipschitz continuous. 
This example shows that the prerequisite of permeability in \cite[Theorem 15]{leoste21}
cannot simply be dropped.
Ultimately, $\Theta$ will be the graph of a continuous 
function $\theta\colon \R\to\R$, 
and therefore a topological submanifold of $\R^2$, which is closed as a set. By Theorems \ref{thm:bdd} and \ref{th:harmony},   the function  $\theta$ cannot be absolutely continuous  or Lipschitz continuous.


 Recall that the function $g_{H}$ was introduced in Section \ref{sec:hoelder-example}.  
For the remainder of this section let $g\colon [0,1]\to \R$ be defined by
\begin{equation}\label{eq:definition-g}
g(x):=\begin{cases}
g_H(x)\,,&x\in[\varepsilon_1,\varepsilon_2],\\
0\,,& \text{else,}
\end{cases}
\end{equation}
where $ \varepsilon_1$ and $\varepsilon_2$ are the smallest and largest  zeros  of $g_H$, respectively. 
Note that $g$ also has an impermeable graph. Indeed, if this were not the case, then one could find a
function $f\colon [0,1]\to \R$ with $f(0)=0$ and $V_0^1(f)<\frac{1}{4}$ such that the set 
$\overline{\{t\in[0,1]\colon f(t)=g(t)\}}$
is countable. Define
\[
\tilde f(x):=\begin{cases}
f(x)\,,&x\in[\varepsilon_1,\varepsilon_2],\\
1/4 \,,& \text{else.}
\end{cases}
\] 
Then $V_0^1(\tilde f)<V_0^1(f)+|f(\epsilon_1)|+|f(\epsilon_2)|+1/4\le 3V_0^1(f)+1/4 <1$ and
\[
\overline{\{t\in[0,1]\colon \tilde f(t)=g(t)\}}\subseteq\overline{\{t\in[0,1]\colon f(t)=g(t)\}}\,,
\]
so the former is countable, a contradiction.

Let
\begin{equation}\label{eq:Aset}
A:=\{(1+x,g(x))\colon x\in[0,1] \}.
\end{equation}  
According to Theorem  \ref{th:harmony}
and Theorem \ref{thm:hoelder-ex},
this implies that there is $c>0$  such that 
\begin{equation}\label{eq:definition-c}
\ell(\gamma)\ge 1+c
\end{equation} 
for every 
path $\gamma$ from $(1,0)$ to $(2,0)$ with countable $\overline{\gamma([0,1])\cap A}$.

\medskip


Next we define a useful representation of the classical Cantor set as a subset of $\R^2$ using two affine transformations. Let $T_1(x,y):=(\frac{x}{3},\frac{y}{2})$ and $T_0(x,y):=(\frac{2+x}{3},\frac{y}{2})$. 
For a number $n\in \mathbb{N}$ let $d(k,n)\in \{0,1\}$ be the $k$-th binary digit, $n=\sum_{k=0}^\infty d(k,n)2^k$,  and
$T_n:=T_{d(0,n)}\circ \dots \circ T_{d(\lfloor\log_2(n)\rfloor,n)} $.  

Note that, if $C\in [0,1]$ is the classical Cantor set, then  
\[
C\times \{0\}=\big([0,1]\times \{0\}\big)\setminus \bigcup_{n\in \mathbb{N}}T_n\big((1,2)\times \{0\}\big)\,.
\]

\begin{lemma}\label{lem:f-lip}
Define  $\tilde{\Theta}:=\bigcup_{n\in \mathbb{N}}T_n(A)$ and $\Theta:=\tilde{\Theta}\cup (C\times\{0\})$. 
Then $\Theta$ is the graph of a Hölder continuous function $\theta\colon[0,1]\to\R$, with Hölder  exponent 
$\beta:=\tfrac{\log(10)}{\log(234)}$.

Further let $f\colon \R\times\{0\}\to \R$,
\[
f((t,0))=\begin{cases}
0\,,& t<0\,,\\
\cantorf(t)\,,& 0\le t\le 1\,,\\ 
1\,, & t>1\,,
\end{cases}
\]
where $\cantorf\colon [0,1]\to \R$ is the classical Cantor (or Devil's) staircase function, 
cf.\cite[p. 252]{thomson2001elementary}.
Then $f$ is Lip\-schitz continuous with respect to $\rho_{\R^2}^\Theta$ restricted to $\R\times\{0\}$.
\end{lemma}

\begin{proof} 
We first show the Lipschitz continuity of (the restriction of) $f$ with respect to $\rho^\Theta_{\R^2}$. Let $t_1,t_2\in \R$. We have to  show that \[|f((t_2,0))-f((t_1,0))|\le \frac{2}{c}\rho^\Theta_{\R^2}((t_1,0),(t_2,0))\,.\]

We concentrate on the interesting case where $t_1,t_2\in C$ with $t_1< t_2$. 
Let $k=\min\{j\ge 1: 1<(t_2-t_1)3^{j}\}+1$ . 
There is precisely one interval $I$ of the form $I=(j3^{-k},(j+1)3^{-k})$ between $t_1$ and $t_2$ and therefore there exists 
$n\in \mathbb{N}$
with $I=\pr_1\circ T_n\big((1,2)\times\{0\}\big)$,  and $k=\lfloor \log_2(n)\rfloor+1$. 

W.l.o.g., $t_1\le (\frac{1}{3})^{k}<2\cdot(\frac{1}{3})^k\le t_2$, such that there exist 
 $s_1\in [0,1]$  and $s_2\in [2,3]$ such that 
$(t_1,0)=T_{2^{k}-1}(s_1,0)$ and
$(t_2,0)=T_{2^{k}-1}(s_2,0)$. 

Let $\eta\colon [0,1]\to \R^2$  be a path from $(t_1,0)$ to $(t_2,0)$ with $\ell(\eta)<\infty$ and such that $\eta ([0,1])\cap \Theta$ has countable closure. 
Write $\hat{\eta}:=(T_{2^k-1})^{-1}\circ \eta$, i.e.,
$\eta= (T_1)^{k}(\hat{\eta})$. 
 Note that $(s_1,0)=\hat{\eta}(0)$ and $(s_2,0)=\hat{\eta}(1)$.
Let $r_0:=\sup\{t\ge 0: \hat\eta_1(t)=1\}$ and
$r_1:=\inf\{t\ge t_0: \hat\eta_1(t)=2\}$. 
We define a path $\gamma$ from $(1,0)$ to $(2,0)$ as the concatenation of 
the straight line from $(1,0)$ to $\hat\eta(r_0)$, the path $\hat\eta|_{[r_0,r_1]}$ and 
the straight line from $\hat\eta(r_1)$ to $(2,0)$. 
Note that $\gamma([0,1])\cap A\subseteq \hat\eta([0,1])\cap A$ and therefore 
the closure of  $\gamma([0,1])\cap A$ is countable. Using \eqref{eq:definition-c}, we get
\begin{align*}
1+c&\le \ell(\gamma)=|\hat\eta(r_0)-(1,0)|+\ell(\hat\eta|_{[r_0,r_1]})+|\hat\eta(r_1)-(2,0)|\\
&=|\hat\eta_2(r_0)|+\ell(\hat\eta|_{[r_0,r_1]})+|\hat\eta_2(r_1)|\\
&\le \|\hat\eta(r_0)-(s_1,0)\|+\ell(\hat\eta|_{[r_0,r_1]})+\|\hat\eta(r_1)-(s_2,0)\|\\
&\le \ell(\hat\eta|_{[0,r_0]})+\ell(\hat\eta|_{[r_0,r_1]})+\ell(\hat\eta|_{[r_1,1]})
=\ell(\hat\eta)\,.
\end{align*}

We can find a path $\bar\gamma\colon [0,1]\to \R^2$ 
with increasing first component like in Proposition \ref{prop:monComponentPath}. Thus $\ell(\bar\gamma)\le \ell(\gamma)$
and $V_0^1(\bar\gamma_2)\le V_0^1(\gamma_2)$. 
We show like in the proof of Theorem \ref{th:harmony}, that 
$\overline{\bar\gamma([0,1])\cap A}$ is countable: By construction, $\bar{\gamma}$ has an at most countable number of vertical line segments, where intersections with the set $A$ that have not already been part of $\overline{\gamma([0,1])\cap A}$, may occur. Since $A$ is the graph of a function, each vertical line segment admits at most one additional intersection point. Let now $\bar{x}$ be a limit point of those additional intersection points. Then $\bar{x}$ is either part of a vertical segment of $\bar{\gamma}$, of which there are at most countably many, or $\bar{x}$ is not part of a vertical segment.
 In the latter case, $\bar{x}$ must be contained in $\gamma([0,1])$, since on the complement of the vertical segments, the images of $\gamma$ and $\bar{\gamma}$ coincide (by Proposition \ref{prop:monComponentPath}) and they both contain $\bar{x}$ because of their continuity. So, in this case, $\bar{x}$ is contained in $\overline{\gamma([0,1])\cap A}$, which is countable. So in both cases, the possible elements of $\overline{\bar\gamma([0,1])\cap A}$ are contained in $\overline{\gamma([0,1])\cap A}$ or in the countable set of intersection points arising from vertical segments of $\bar{\gamma}$. Hence $\overline{\bar\gamma([0,1])\cap A}$ is countable. Therefore \eqref{eq:definition-c} holds with $\bar\gamma$ instead of $\gamma$.
 Using Proposition \ref{prop:estimate-len-var} we get
(note that $T_{2^{k}-1}=(T_1)^{k}$)
\begin{align*}
\ell (\eta)&\ge V_0^1 \big( \eta_2 \big)=\left(\frac{1}{2}\right)^k V_0^1 \big( \hat{\eta}_2\big)
\ge\left(\frac{1}{2}\right)^k V_0^1 \big( \bar\gamma_2 \big)\\
&\ge \left(\frac{1}{2}\right)^k  \Big(\ell(\bar\gamma)-V_0^1 \big( \bar\gamma_1  \big)\Big)
\ge\left(\frac{1}{2}\right)^k  \big(\ell(\bar\gamma)-1\big)\ge \left(\frac{1}{2}\right)^k c\,.
\end{align*}
It follows that $\rho^\Theta_{\R^2}\big((t_1,0),(t_2,0)\big)\ge \left(\frac{1}{2}\right)^k c$.

On the other hand, we have $\big|f((t_2,0))-f((t_1,0))\big|\le  \left(\frac{1}{2}\right)^{k-1}$, so that 
\[
\big|f((t_2,0))-f((t_1,0))\big|\le\frac{2}{c}\rho^\Theta_{\R^2}\big((t_1,0),(t_2,0)\big)\,.
\]

Since $g(0)=g(1)=0$, $\Theta$ is the graph of a continuous function $\theta\colon [0,1]\to \R$ by construction. Now we show the claim about the Hölder continuity of $\theta$. 
For this assume that $g$ is Hölder continuous with exponent $\alpha$ and constant $C_g$, and let 
$t_1,t_2\in [0,1]$. 
If $t_1,t_2\in C$, then $\theta(t_1)=\theta(t_2)=0$ and there is nothing to show.
Next consider the case where $t_1,t_2$ are in the same connected 
component of $\R\setminus C$. That is, $(t_1,0),(t_2,0)\in T_n((1,2)\times\{0\})$ for some $n\in \N$.
Again we may concentrate on the case $n=2^k-1$, so that w.l.o.g.
$\frac{1}{3}(\frac{1}{3})^{k-1}<t_1<t_2<\frac{2}{3}\cdot(\frac{1}{3})^{k-1}$, and 
$1<s_1:=\pr_1\big(T_n^{-1}(t_1,0)\big)<s_2:=\pr_1\big(T_n^{-1}(t_2,0)\big)<2$.
By Theorem \ref{thm:hoelder-ex} and \eqref{eq:definition-g},   we have 
\[
|g(s_1)-g(s_2)|\le C_g|s_1-s_2|^{\beta }\,,
\]
so that 
\begin{align*}
|\theta(t_1)-\theta(t_2)|&=2^{-k}|g(s_1)-g(s_2)|\le 2^{-k}C_g|s_1-s_2|^{\beta }\\
&= 2^{-k}C_g(3^k|t_1-t_2|)^{\beta }=\left(\frac{3^{{\beta }}}{2}\right)^kC_g|t_1-t_2|^{\beta }\,.
\end{align*}
Now since  ${\beta }\le \tfrac{\log(2)}{\log(3)}$, we have  $\tfrac{3^{{\beta }}}{2}\le 1$ and 
$|\theta(t_1)-\theta(t_2)|\le C_g|t_1-t_2|^{\beta }$. 
Next suppose that $t_1\in C$, $t_2\in [0,1]\setminus C$.  
Denote by $(u,v)$  the
largest open interval with 
$t_2\in (u,v)   \subseteq [0,1]\setminus C$. If  $t_1<t_2$, then by the continuity of $\theta$
\[
|\theta(t_1)-\theta(t_2)|=|0-\theta(t_2)|=|\theta(u )-\theta(t_2)|\le C_g |u -t_2|^\beta\le C_g |t_1-t_2|^\beta\,,
\]
and in the same way we can treat the case  $t_2<t_1$.

Next we consider the case where $t_1$ and $t_2$, $t_1<t_2$, lie in different connected components 
 $(u_1,v_1)$ and $(u_2,v_2)$  of $[0,1]\setminus C$, respectively. 
 Then 
 \begin{align*}
 |\theta(t_1)-\theta(t_2)|
 &=|\theta(t_1)-\theta(v_1)+\theta(u_2)-\theta(t_2)|\\
 &\le |\theta(t_1)-\theta(v_1)|+|\theta(u_2)-\theta(t_2)|\\
 &\le C_g|t_1-v_1|^\beta+C_g|u_2-t_2|^\beta\\
 &\le 2^{1-\beta}C_g\big(|t_1-v_1|+|u_2-t_2|\big)^\beta
 \le 2^{1-\beta}C_g|t_1-t_2|^\beta\,,
 \end{align*} 
where we used the special case Hölder inequality $(x+y)^q\le 2^{q-1}(x^q+y^q)$ for all $x,y\in [0,\infty)$, $q\in [1,\infty)$ with $x=|t_1-v_1|^\beta$,
$y=|u_2-t_2|^\beta$ and $q=1/\beta$. 
\end{proof}


\begin{theorem}\label{*thmintr}
There exists a continuous function $f\colon\R^2\to \R$ which is intrinsically Lipschitz continuous
on $\R^2 \setminus \Theta$, where $\Theta$ is the graph of a Hölder continuous function $\theta$, but $f$ is not Lipschitz with respect to the Euclidean  metric.
\end{theorem}

\begin{proof} The function $f$ constructed  in Lemma \ref{lem:f-lip}   is Lipschitz continuous on $\R\times \{0\}$ with respect to $\rho^\Theta_{\R^2}$.

By the Kirszbraun extension theorem, cf.~\cite[Theorem 2.10.43]{federer1969} (or the elementary special case 
\cite[2.10.44]{federer1969}), there exists a Lipschitz extension of $f$ on  $\R^2$
which we again denote by $f$.
Hereby `Lipschitz' means with respect to the metric $\rho_{\R^2}^\Theta$. But since
$\rho_{\R^2}^\Theta\le \rho_{\R^2\setminus \Theta}$, $f$ is also Lipschitz with respect to the intrinsic metric on 
$\R^2\setminus\Theta$. Restricted to the domain $[0,1]\times\{0\}$, $f$ equals  the Cantor staircase function, which is not Lipschitz with respect to the  Euclidean  metric. It remains to show the continuity of $f$ with respect to the  Euclidean metric:

To that end, take a sequence $(x_n)_n$, converging to $x\in\R^2$. Note that, being the  graph of a continuous function on $[0,1]$, $\Theta$ is closed.\medskip

{\em Case 1:} $x\notin \Theta$.\\
Then there is $r>0$ and $N\in\N$ such that for all $n\geq N$, $x_n\in B_r(x)\subseteq\R^2\setminus \Theta$. It follows that in this ball  $\rho^\Theta_{\R^2}$ equals the  Euclidean metric. So $f$ is Lipschitz
therein  with respect to both metrics and in particular $f$ is continuous in $x$.\medskip

{\em Case 2:} $x\in \Theta$.\\
Then w.l.o.g., we can assume that all the $x_n$ lie above or on the graph of $\theta$ and $\pr_1(x_n)\nearrow\pr_1(x)$. Let $v_n:=\big(\pr_1(x_n), \theta(\pr_1(x_n))\big)$. Then, as $x_n\to x$, also $x_n - v_n\to 0$. Also, since $\theta$ is continuous, 
\begin{align*}
M(x_n,x):=\max_{\pr_1(x_n)\leq s\leq \pr_1(x)}(\theta(s)-\theta(\pr_1(x_n)))\to 0.
\end{align*}
Now define
\begin{align*}
w_n:=\left(\pr_1(x_n), \max\left\{\pr_2(x_n),\tfrac{1}{n}+\theta(\pr_1(x_n))+M(x_n,x)\right\}\right)
\end{align*}
and
\begin{align*}
z_n:=\left(\pr_1(x), \max\left\{\pr_2(x_n),\tfrac{1}{n}+\theta(\pr_1(x_n))+M(x_n,x)\right\} \right).
\end{align*}
Then, the path $\gamma_n$, defined as the concatenation of line segments $v_n\to w_n\to z_n\to x$ is in $\R^2\setminus\Theta$ apart from $v_n$ and $x$ and contains $x_n$ in its image. Forms of such paths are illustrated in Figure \ref{fig:f-cont}.
\begin{figure}
\begin{center}
\begingroup%
  \makeatletter%
  \providecommand\color[2][]{%
    \errmessage{(Inkscape) Color is used for the text in Inkscape, but the package 'color.sty' is not loaded}%
    \renewcommand\color[2][]{}%
  }%
  \providecommand\transparent[1]{%
    \errmessage{(Inkscape) Transparency is used (non-zero) for the text in Inkscape, but the package 'transparent.sty' is not loaded}%
    \renewcommand\transparent[1]{}%
  }%
  \providecommand\rotatebox[2]{#2}%
  \newcommand*\fsize{\dimexpr\f@size pt\relax}%
  \newcommand*\lineheight[1]{\fontsize{\fsize}{#1\fsize}\selectfont}%
  \ifx\svgwidth\undefined%
    \setlength{\unitlength}{189.40857757bp}%
    \ifx\svgscale\undefined%
      \relax%
    \else%
      \setlength{\unitlength}{\unitlength * \real{\svgscale}}%
    \fi%
  \else%
    \setlength{\unitlength}{\svgwidth}%
  \fi%
  \global\let\svgwidth\undefined%
  \global\let\svgscale\undefined%
  \makeatother%
  \begin{picture}(1,0.57448679)%
    \lineheight{1}%
    \setlength\tabcolsep{0pt}%
    \put(0,0){\includegraphics[width=\unitlength,page=1]{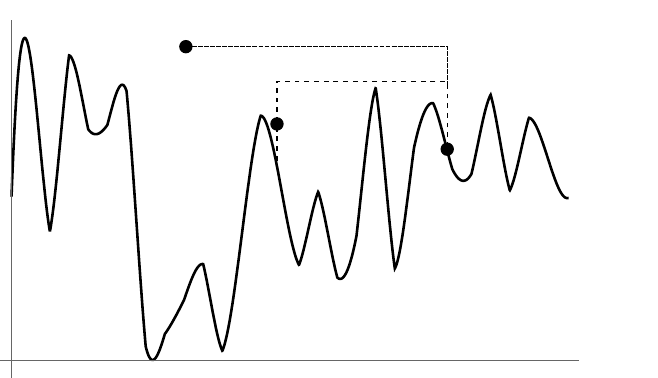}}%
    \put(0.66280821,0.3043549){\makebox(0,0)[t]{\lineheight{1.25}\smash{\begin{tabular}[t]{c}$x$\end{tabular}}}}%
    \put(0.3413213,0.53280173){\makebox(0,0)[t]{\lineheight{1.25}\smash{\begin{tabular}[t]{c}$x_n=w_n$\end{tabular}}}}%
    \put(0.47655026,0.37408775){\makebox(0,0)[t]{\lineheight{1.25}\smash{\begin{tabular}[t]{c}$x_{n'}$\end{tabular}}}}%
    \put(0.47650297,0.46437078){\makebox(0,0)[t]{\lineheight{1.25}\smash{\begin{tabular}[t]{c}$w_{n'}$\end{tabular}}}}%
    \put(0.73186487,0.45887444){\makebox(0,0)[t]{\lineheight{1.25}\smash{\begin{tabular}[t]{c}$z_{n'}$\end{tabular}}}}%
    \put(0.72273394,0.52576597){\makebox(0,0)[t]{\lineheight{1.25}\smash{\begin{tabular}[t]{c}$z_{n}$\end{tabular}}}}%
    \put(0,0){\includegraphics[width=\unitlength,page=2]{continuous-neu.pdf}}%
    \put(0.24820424,0.13664154){\makebox(0,0)[t]{\lineheight{1.25}\smash{\begin{tabular}[t]{c}$v_{n}$\end{tabular}}}}%
    \put(0,0){\includegraphics[width=\unitlength,page=3]{continuous-neu.pdf}}%
    \put(0.47237184,0.31243199){\makebox(0,0)[t]{\lineheight{1.25}\smash{\begin{tabular}[t]{c}$v_{n'}$\end{tabular}}}}%
    \put(0.9012242,0.28616852){\makebox(0,0)[t]{\lineheight{1.25}\smash{\begin{tabular}[t]{c}$\theta$\end{tabular}}}}%
  \end{picture}%
\endgroup%
\end{center}
\caption{Illustration of the proof of continuity of $f$ with respect to the Euclidean  metric on 
$\R^2$ .}\label{fig:f-cont}
\end{figure}

Further,
\begin{align*}
\ell(\gamma_n)=&\max\left\{\pr_2(x_n)-\theta(\pr_1(x_n)),\tfrac{1}{n}+M(x_n,x)\right\}+(\pr_1(x)-\pr_1(x_n))\\
&\ +\max\left\{\pr_2(x_n)-\theta(\pr_1(x)),\tfrac{1}{n}+\theta(\pr_1(x_n))-\theta(\pr_1(x))+M(x_n,x)\right\},
\end{align*}
and $\ell(\gamma_n)\to 0$ as all terms on the r.h.s.~tend to $0$. Since the image of $\gamma_n$ has at most  two points in common with $\Theta$, it follows that
$\rho^\Theta_{\R^2}(x_n,x)\leq \ell(\gamma_n)\to 0$.  By the Lipschitz continuity with respect to $\rho^\Theta_{\R^2}$ with Lipschitz constant, say $L$, we have
$$|f(x_n)-f(x)|\leq L\rho^\Theta_{\R^2}(x_n,x)\to 0\quad\text{as}\quad x_n\to x,$$
and we obtain the continuity of $f$.
\end{proof}

\section*{Acknowledgements}

Gunther Leobacher and Alexander Steinicke are grateful to Dave L.~Renfro for connecting them with 
Zolt\'an Buczolich.


\subsection*{Author's addresses}

\noindent {Zolt\'an Buczolich}
{Department of Analysis, ELTE E\"otv\"os Lor\'and University, 
P\'azm\'any P\'eter s\'e\-t\'any 1/c, H-1117 Budapest, Hungary}
\\{\tt zoltan.buczolich@ttk.elte.hu}\\

\noindent Gunther Leobacher, Institute of Mathematics and Scientific Computing, University of Graz. 
Heinrichstraße 36, 8010 Graz, Austria. \\{\tt gunther.leobacher@uni-graz.at}\\

\noindent Alexander Steinicke, Institute of Mathematics, 
University of Rostock,
Ulmenstra\ss e 69, Haus 3, 18051 Rostock, Germany. \\{\tt alexander.steinicke@uni-rostock.de}\smallskip

\noindent Alexander Steinicke, Institute of Applied Mathematics, 
Montanuniversitaet Le\-o\-ben.
Peter-Tunner-Straße 25/I, 8700 Leoben, Austria. \\{\tt alexander.steinicke@unileoben.ac.at}

\end{document}